\documentclass[12pt]{amsart}

\usepackage{hyperref}
\usepackage{verbatim}
\usepackage{fullpage}
\usepackage{amsfonts}

\usepackage{graphicx}
\usepackage{amsmath, amsthm, amssymb}
\usepackage{fancybox}
\newcommand{\swoosh}{\includegraphics[width=0.35in]{swoosh.pdf}}

\renewcommand{\tilde}{\widetilde}

\newcommand{\R}{\mathbb{R}}
\newcommand{\C}{\mathbb{C}}
\newcommand{\N}{\mathbb{N}}
\newcommand{\Z}{\mathbb{Z}}

\newcommand{\Hil}{\mathcal{H}}
\newcommand{\Kil}{\mathcal{K}}

\newcommand{\eps}{\varepsilon}
\newcommand{\la}{\lambda}
\newcommand{\ve}{\varepsilon}

\DeclareMathOperator{\lspan}{span}

\providecommand{\norm}[1]{\lVert#1\rVert}

\newcounter{Theorem}

\numberwithin{equation}{section}
\numberwithin{Theorem}{section}

\theoremstyle{plain} %% This is the default, anyway
\newtheorem{thm}[Theorem]{Theorem}
\newtheorem{cor}[Theorem]{Corollary}
\newtheorem{lem}[Theorem]{Lemma}

\theoremstyle{definition}
\newtheorem{defn}[Theorem]{Definition}

\theoremstyle{remark}
\newtheorem{remark}{Remark}[section]
\newtheorem{ex}[Theorem]{Example}

\begin{document}

\title{The Schur-Horn Theorem for Operators with Finite Spectrum}

\author{Marcin Bownik}

\address{Department of Mathematics, University of Oregon, Eugene, OR 97403--1222, USA}

\email{mbownik@uoregon.edu}

\author{John Jasper}

\address{Department of Mathematics, University of Missouri, Columbia, MO 65211--4100, USA}

\email{jasperj@missouri.edu}

\keywords{diagonals of self-adjoint operators, the Schur-Horn theorem, the Pythagorean theorem, the Carpenter theorem}

\thanks{
This work was partially supported by a grant from the Simons Foundation (\#244422 to Marcin Bownik).
The second author was supported by NSF ATD 1042701}

\subjclass[2000]{Primary: 42C15, 47B15, Secondary: 46C05}
\date{\today}

\begin{abstract}
 We characterize the set of diagonals of the unitary orbit of a self-adjoint operator with a finite spectrum. Our result extends the Schur-Horn theorem from a finite dimensional setting to an infinite dimensional Hilbert space analogous to Kadison's theorem for orthogonal projections \cite{k1,k2} and the second author's result for operators with three point spectrum \cite{jj}. 
\end{abstract}

\maketitle

\section{Introduction}

% Copied from npt-pnas paper

The classical Schur-Horn theorem \cite{horn, schur} characterizes diagonals of self-adjoint (Hermitian) matrices with given eigenvalues. It can be stated as follows, where $\mathcal H_N$ is an $N$ dimensional Hilbert space over $\R$ or $\C$, i.e., $\mathcal H_N=\R^N$ or $\C^N$.

\begin{thm}[Schur-Horn theorem]\label{horn} 
Let $\{\lambda_{i}\}_{i=1}^{N}$ and $\{d_{i}\}_{i=1}^{N}$ be real sequences in nonincreasing order. There exists a self-adjoint operator $E:\mathcal H_N \to\mathcal H_N$ with eigenvalues $\{\lambda_{i}\}$ and diagonal $\{d_{i}\}$
if and only if 
\begin{equation}\label{horn1}
\sum_{i=1}^{N}d_i =\sum_{i=1}^{N}\lambda_{i} \quad\text{ and }\quad \sum_{i=1}^{n}d_{i} \leq \sum_{i=1}^{n}\lambda_{i} \text{ for all } 1\le n \le N.
\end{equation}
\end{thm}

The necessity of \eqref{horn1} is due to Schur \cite{schur}, and the sufficiency of \eqref{horn1} is due to Horn \cite{horn}. It should be noted that \eqref{horn1} can be stated in the equivalent convexity condition
\begin{equation}\label{horn2}
(d_1,\ldots,d_N) \in \operatorname{conv} \{ (\lambda_{\sigma(1)},\ldots,\lambda_{\sigma(N)}): \sigma \in S_N\}.
\end{equation}
This characterization has attracted  significant interest and has been generalized in many remarkable ways. Some major milestones are the Kostant convexity theorem \cite{ko} and the convexity of moment mappings in symplectic geometry \cite{at, gs1, gs2}. Moreover, the problem of extending Theorem \ref{horn} to an infinite dimensional Hilbert space $\mathcal H$ has attracted a great deal of interest.

Neumann \cite{neu} gave an infinite dimensional version of the Schur-Horn theorem phrased in terms of $\ell^\infty$-closure of the convexity condition \eqref{horn2}. Neumann's result can be considered an initial, albeit somewhat crude, solution of this problem. The first fully satisfactory progress 
%to this problem 
was achieved by Kadison. In his influential work \cite{k1, k2} Kadison discovered a characterization of diagonals of orthogonal projections acting on $\mathcal H$. The work by Gohberg and Markus \cite{gm} and Arveson and Kadison \cite{ak}
extended the Schur-Horn Theorem \ref{horn} to positive trace class operators. This has been further extended to compact positive operators by Kaftal and Weiss \cite{kw}. These results are stated in terms of majorization inequalities as in \eqref{horn1}. Other notable progress includes the work of Arveson \cite{a} on diagonals of normal operators with finite spectrum.
Moreover, Antezana, Massey, Ruiz, and Stojanoff \cite{amrs} refined the results of Neumann \cite{neu}, and Argerami and Massey \cite{am, am2} studied extensions to II$_1$ factors.
For a detailed survey of recent progress on infinite Schur-Horn majorization theorems and their connections to operator ideals we refer to the paper of Kaftal and Weiss \cite{kw0}.

The authors \cite{mbjj} have recently shown a variant of the Schur-Horn theorem for a class of locally invertible self-adjoint operators on $\mathcal H$. This result was used to characterize sequences of norms of a frame with prescribed lower and upper frame bounds. The second author \cite{jj} has extended Kadison's result \cite{k1,k2} to characterize the set of diagonals of the unitary orbit of a self-adjoint operator with three points in the spectrum. In this work we shall continue this line of research by studying self-adjoint operators with finite spectrum. 
%END copy

There are two distinct extensions of the Schur-Horn theorem for operators with finite spectrum. The case when the multiplicities of eigenvalues are not prescribed was already considered by the authors in \cite{mbjj2}.
While the main result in \cite{mbjj2} provides a satisfactory description of possible diagonals of operators with finite spectrum, it is far from describing diagonals of the unitary orbit of such operators. In other words, a fully satisfactory Schur-Horn theorem should characterize the diagonals of operators with given eigenvalues and their corresponding multiplicities. This leads to the second more complete variant of the Schur-Horn theorem.
Before we state the full theorem, we need to set up some convenient notation.

\begin{defn}\label{conv}
 Let $\{A_j\}$ be a finite increasing sequence in $\R$, and let $\{N_j\}$ be a sequence in $\N \cup \{\infty\}$ (with the same index set) that takes the value of $\infty$ {\bf at least twice}. Without loss of generality we shall assume that the combined sequence is reindexed as $\{(A_j,N_j)\}_{j=-m}^{n+p+1}$ for some $m,n,p \in \N_0$ and that
\[
N_{0}=N_{n+1}=\infty \qquad\text{and}\qquad N_j<\infty\text{ for }j<0\text{ and } j>n+1.
\]
For simplicity we shall assume that $A_0=0$ and $A_{n+1}=B$.

Let $\{d_i\}_{i\in I}$ be a sequence in $[A_{-m},A_{n+p+1}]$. For each $\alpha \in (0,B)$, define 
\[
C(\alpha) = \sum_{d_{i}<\alpha}d_{i}\quad\text{and}\quad D(\alpha)=\sum_{d_{i}\geq \alpha}(B-d_{i}).
\]
Since the above series may have both positive and negative terms, we shall follow the convention that $C(\alpha)=\infty$ or $D(\alpha)=\infty$, if the corresponding series is not absolutely convergent. Thus, $C(\alpha)<\infty$ means that the series $\sum_{d_{i}<\alpha}d_{i}$ is absolutely convergent.

Let $E$ be a bounded operator on a Hilbert space $\mathcal H$. For $\lambda \in \C$ define
\[
m_E(\lambda)= \dim \ker (E-\lambda).
\]
We say that an operator $E$ has an eigenvalue-multiplicity list $\{(A_j,N_j)\}_{j=-m}^{n+p+1}$ if its spectrum $\sigma(E)=\{A_{-m},\ldots, A_{n+p+1} \}$ and $m_E(A_j)=N_j$ for all $-m \le j \le n+p+1$.
\end{defn}

We are now ready to state the main result of this paper for operators with at least two infinite multiplicity eigenvalues. The corresponding result with one infinite multiplicity is less involved, see Theorem \ref{frhorn}, whereas the case of all finite multiplicities is the classical Schur-Horn theorem.

\begin{thm}
\label{fullthm}
Let $\{(A_j,N_j)\}_{j=-m}^{n+p+1}$, $m,n,p \in \N_0$, be a sequence as in Definition \ref{conv}, and let $\{d_{i}\}_{i\in I}$ be a sequence in $[A_{-m},A_{n+p+1}]$. There exists a self-adjoint operator $E$ with diagonal $\{d_{i}\}_{i\in I}$ and the eigenvalue-multiplicity list $\{(A_j,N_j)\}_{j=-m}^{n+p+1}$ if and only if the following three conditions hold:
\begin{enumerate}

\item (lower exterior majorization) 
for all $r=-m, \ldots, 0$,
\begin{equation}\label{fulllow}
\sum_{d_i \le A_r} (A_r-d_i) \le \sum_{j=-m}^{r-1} (A_r-A_j)N_j.
\end{equation}

\item either we have:
\begin{itemize}
\item (non-summability) $C(B/2)=\infty$ or $D(B/2)=\infty$, and these are the only possibilities if $N_j$ takes the value of $\infty$ more than twice, or 
\item (interior majorization) $C(B/2)<\infty$ and $D(B/2)<\infty$ $($and thus $C(\alpha)<\infty$ and $D(\alpha)< \infty$ for all $\alpha \in (0,B)$$)$, and there exists $k\in \Z$ such that the following three conditions hold:
\begin{align}
\label{fullinf}
|\{i\in I: d_i < B/2\}| &= |\{ i\in I: d_i \ge B/2 \}| = \infty,
\\
\label{fulltrace} 
C(B/2) -D(B/2) & =  \sum_{\genfrac{}{}{0pt}{}{j=-m}{j \not= 0,n+1}}^{n+p+1} A_j N_j +k B,
\\
\label{fullmaj} 
(B-A_r)C(A_r) +A_r D(A_r) & \ge
(B-A_r)\sum_{\genfrac{}{}{0pt}{}{j=-m}{j \not= 0}}^{r} A_j N_j + A_r \sum_{\genfrac{}{}{0pt}{}{j=r+1}{j \ne n+1}}^{n+p+1} (B-A_j)N_j
\\ 
\notag
& \hskip5cm \text{for all } r=1,\ldots,n,
\end{align}
\end{itemize}

\item (upper exterior majorization) for all $r=n+1, \ldots, n+p+1$,
\begin{equation}\label{fullup}
\sum_{d_i \ge A_r} (d_i-A_r) \le \sum_{j=r+1}^{n+p+1} (A_j-A_r)N_j.
\end{equation}
\end{enumerate}
\end{thm}

We remark that the trace condition \eqref{fulltrace} makes sense only when all $N_j<\infty$ for $j=1,\ldots,n$. In fact, the assumption that $C(B/2)<\infty$ and $D(B/2)<\infty$ actually forces this property in light of Theorem \ref{tt}. In other words, the interior majorization subcase of Theorem \ref{fullthm} can only happen when there are exactly two infinite multiplicities: $N_0=N_{n+1}=\infty$.

The proof of Theorem \ref{fullthm} occupies most of the paper, and it is broken into several parts. Section 2 recalls the most fundamental results used in this paper such as Kadison's theorem and the ``moving toward $0$-$1$'' lemma, which were extensively employed in authors' earlier work \cite{mbjj, jj}. By far the easiest part of the proof of Theorem \ref{fullthm} is the necessity of exterior majorization in Section 3. In contrast, the necessity of interior majorization is much more complicated, and it splits into two stages. First, we establish the trace condition \eqref{fulltrace}. At the same time we show that the non-summability subcase of (ii) must necessarily happen when more than two eigenvalues have infinite multiplicities. Second, we establish the majorization inequalities \eqref{fullmaj}.

Section 5 shows sufficiency of interior majorization in the special case $m=p=0$, i.e., when exterior majorization is not present. We introduce an important concept of Riemann interior majorization which requires ordering of a diagonal sequence $\{d_i\}_{i\in I}$, and which resembles classical majorization as in \cite{ak, kw0, kw}. In contrast, Lebesgue interior majorization does not require any ordering and is the main invention of the paper. In the crucial case, when $\{d_i\}_{i\in I}$ can be put in nondecreasing order indexed by $\Z$, these two concepts coincide. The proof of this result is elementary, albeit long. The sufficiency of Riemann interior majorization, which plays a central part in the paper, requires an involved combinatorial argument employing machinery from Section 2. Finally, Theorem \ref{inthorn} deals with sequences that satisfy Lebesgue majorization but do not conform to Riemann majorization.

Section 6 shows the sufficiency of exterior majorization in the case when there is exactly one eigenvalue with infinite multiplicity, which is not covered by Theorem \ref{fullthm}. We show that the special case when only either lower or upper exterior majorization is present can be conveniently and swiftly dealt using interior majorization. We also establish a ``decoupling'' lemma, which plays an important role in the rest of our arguments. In short, this lemma enables us to modify our diagonal sequence $\{d_i\}_{i\in I}$ into two separate sequences satisfying lower and upper exterior majorization, respectively. A similar technique is used in Section 7, which shows the sufficiency part of Theorem \ref{fullthm}. We use the decoupling lemma to obtain three modified diagonal sequences satisfying lower and upper exterior, and interior majorization, resp. This process requires careful analysis of resulting diagonal sequences belonging to the same unitary orbit of a suitable self-adjoint operator. As a consequence we obtain two sufficiency results corresponding to the two subcases of part (ii), thus completing the proof of Theorem \ref{fullthm}. 

Finally, in Section 8 we consider a converse problem of characterizing spectra of operators with a fixed diagonal. While Theorem \ref{fullthm} does not resemble the Schur-Horn Theorem in any obvious way, its converse counterpart Theorem \ref{La} does. Given a diagonal sequence $\{d_i\}_{i\in I}$, which satisfies some natural summability conditions, we consider the set $\Lambda_N(\{d_i\})$ of possible lists of $N$ eigenvalues of operators with such diagonals, see \eqref{LN}. Theorem \ref{La} states that $\Lambda_N(\{d_i\})$ has a very special structure. It is a union of $N$, or $N-1$ if $\{d_i\}$ is a diagonal of a projection, upper subsets of constant trace each having a unique minimal element with respect to the majorization order \cite{moa}. This is in close analogy with the Schur-Horn Theorem \ref{horn} which can be restated as follows: the set of possible lists of eigenvalues of operators with fixed diagonal $\{d_i\}_{i=1}^N$ is an upper set with a minimal element $\{d_i\}_{i=1}^N$.

\section{Preliminaries}

The Schur-Horn theorem and its extensions \cite{ak, kw} are usually stated with eigenvalues listed in nonincreasing order indexed by $\N$. However, if we insist on arranging diagonal entries into a nondecreasing sequence, then we should instead use $-\N$ as a part of the indexing set. This leads to two different formulation of the Schur-Horn theorem for finite rank positive operators, see \cite[Theorems 2.1 and 2.2]{mbjj2}. The main innovation here is that we do not require a sequence $\{d_i\}$ to be globally monotone. This allows the possibility that $\{d_i\}$ has infinitely many positive terms and some zero terms. At the same time it also gives us flexibility in arranging small diagonal terms.

\begin{thm}\label{horn-ninc}
Let $\{\lambda_{i}\}_{i=1}^{N}$ be a positive nonincreasing sequence. Let $\{d_{i}\}_{i=1}^{\infty}$ be a nonnegative sequence such that:
\begin{enumerate}
\item
 $d_i \le d_{N}$ for $i\geq N+1$,
\item the subsequence $\{d_i\}_{i=1}^{N}$ is nonincreasing. 
\end{enumerate}

There exists a positive rank $N$ operator $E$ on a Hilbert space $\Hil$ with (positive) eigenvalues $\{\lambda_{i}\}_{i=1}^{N}$ and diagonal $\{d_{i}\}_{i=1}^{\infty}$ if and only if
\begin{equation}
\label{frmaj1}
\sum_{i=n}^{\infty}d_{i} \ge \sum_{i=n}^{N}\lambda_{i} \text{ for all } 1\le n\leq N, \text{ with equality when }n=1.
\end{equation}
\end{thm}

\begin{thm}\label{horn-ndec}
Let $\{\lambda_{i}\}_{i=1}^{N}$ be a positive 
nondecreasing sequence. Let $\{d_{i}\}_{i=-\infty}^{N}$ be a nonnegative sequence such that:
\begin{enumerate}
\item
 $d_i \le d_{1}$ for $i\le 0$,
\item the subsequence $\{d_i\}_{i=1}^{N}$ is nondecreasing. 
\end{enumerate}

There exists a positive rank $N$ operator $E$ on a Hilbert space $\Hil$ with (positive) eigenvalues $\{\lambda_{i}\}_{i=1}^{N}$ and diagonal $\{d_{i}\}_{i=-\infty}^{N}$ if and only if
\begin{equation}
\label{frmaj}
\sum_{i=-\infty}^{n}d_{i} \ge \sum_{i=1}^{n}\lambda_{i}  \qquad \text{ for all } 1\leq n \le N, \text{ with equality when }n=N.
\end{equation}
\end{thm}

Later we shall see yet another variant of the Schur-Horn theorem for positive finite rank operators, Theorem \ref{pfrhorn}, which does not rely on any particular way of ordering of diagonal entries as in Theorems \ref{horn-ninc} and \ref{horn-ndec}. Moreover, we shall also extend this to a general Schur-Horn theorem for finite rank (not necessarily positive) self-adjoint operators, see Theorem \ref{frhorn}. 
We will also make an extensive use of Kadison's theorem \cite{k1,k2} which characterizes diagonals of orthogonal projections.

\begin{thm}[Kadison]
\label{Kadison} Let $\{d_{i}\}_{i\in I}$ be a sequence in $[0,1]$ and $\alpha\in(0,1)$. Define
\[
C(\alpha)=\sum_{d_{i}<\alpha}d_{i}, \qquad D(\alpha)=\sum_{d_{i}\geq \alpha}(1-d_{i}).\]
There exists an orthogonal projection on $\ell^2(I)$ with diagonal $\{d_{i}\}_{i\in I}$ if and only if either:
\begin{enumerate}
\item $C(\alpha)=\infty$ or $D(\alpha)=\infty$, or
\item $C(\alpha)<\infty$ and $D(\alpha)<\infty$, and
\begin{equation}
\label{kadcond} 
C(\alpha)-D(\alpha)\in\Z.
\end{equation}
\end{enumerate}
\end{thm}

\begin{remark}\label{rpart}
Note that if there exists a partition of $I=I_{0}\cup I_{1}$ such that
\begin{equation}\label{partition}
\sum_{i\in I_{0}}d_{i}<\infty \quad\text{and}\quad\sum_{i\in I_{1}}(1-d_{i})<\infty,
\end{equation}
then for all $\alpha\in (0,1)$ we have $C(\alpha)<\infty$ and $D(\alpha)<\infty$ and
\[
\bigg(\sum_{i\in I_{0}} d_{i} - \sum_{i\in I_{1}}(1-d_{i}) \bigg) - (C(\alpha)-D(\alpha))\in\Z.
\]
Thus, in the presence of a partition satisfying \eqref{partition},
\[
\bigg(\sum_{i\in I_{0}} d_{i} - \sum_{i\in I_{1}}(1-d_{i}) \bigg) \in\Z
\]
is a necessary and sufficient condition for a sequence to the be the diagonal of a projection. We will find use for these more general partitions in the sequel.
\end{remark}

The following ``moving toward $0$-$1$'' lemma plays a key role in our arguments. Lemma \ref{movelemma} is simply a concatenation of \cite[Lemmas 4.3 and 4.4]{mbjj}.

\begin{lem}\label{movelemma}
 Let $\{d_{i}\}_{i\in I}$ be a bounded sequence in $\R$ and let $A,B\in\R$ with $A<B$. Let $I_0, I_1 \subset I$ be two disjoint finite subsets such that $\{d_{i}\}_{i\in I_{0}}$ and $\{d_{i}\}_{i\in I_{1}}$ are in $[A,B]$, and $\max\{d_{i}: i \in I_0\}\leq\min\{d_{i}: i \in I_1\}$. Let $\eta_{0}\geq 0$ and
\[
\eta_{0}\leq\min\bigg\{\sum_{i\in I_0} (d_{i}-A),\sum_{i\in I_1}
(B-d_{i}) \bigg\}.\]
(i) There exists a sequence $\{\tilde d_{i}\}_{i\in I}$
satisfying
\begin{align}
\label{ops0}
\tilde{d}_i = d_i &\quad\text{for } i \in I \setminus (I_0 \cup I_1),
\\
\label{ops1}
A\leq \tilde d_i \leq d_i \quad i\in I_0,
&\quad\text{and}\quad
B\geq \tilde d_i \ge d_i, \quad i\in I_1,
\\
\label{ops2}
\eta_0+\sum_{i\in I_0}(\tilde{d}_{i}-A) =\sum_{i\in I_0}(d_{i}-A)
&\quad\text{and}\quad 
\eta_0+\sum_{i\in I_1} (B-\tilde{d}_{i})=\sum_{i\in I_1} (B-d_{i}).
\end{align}
(ii) For any self-adjoint operator $\tilde E$ on $\Hil$ with diagonal $\{\tilde d_{i}\}_{i\in I}$,
there exists an operator $E$ on $\Hil$ unitarily equivalent to $\tilde E$ with diagonal $\{d_{i}\}_{i\in I}$.
\end{lem}

\section{Necessity of exterior majorization}

In this section we will show the necessity of the exterior majorizations in Theorem \ref{fullthm}. This is a consequence of the following two elementary results. Theorem \ref{ext} establishes a majorization for operators with discrete spectrum in the lower part.

\begin{thm}\label{ext}
Suppose that $E$ is a self-adjoint operator on a Hilbert space $\mathcal H$ with 
\[
\sigma(E) \subset \{\la_1,\ldots,\la_{m-1}\} \cup [\la_m,\infty),
\]
where $\la_1 < \ldots< \la_m$, $m\ge 2$. Let $\{e_{i}\}_{i\in I}$ be an orthonormal basis for $\Hil$ and $d_{i}=\langle Ee_{i},e_{i}\rangle$. Then, for any $r=2,\ldots, m$,
\begin{equation}\label{ext1}
\sum_{i\in I,\ d_i \le \la_r} (\la_r - d_i) \le \sum_{j=1}^{r-1} (\la_r-\la_j)N_j,
\qquad\text{where }N_j=m_E(\la_j).
\end{equation}
\end{thm}

\begin{proof}
For a fixed $r=2,\ldots, m$ we decompose $E= \la_1 P_1 + \ldots \la_{r-1}P_{r-1}+\tilde E$, where $P_1,\ldots,P_{r-1}$ are mutually orthogonal projections onto eigenspaces with eigenvalues $\la_1,\ldots,\la_{r-1}$, resp.  Define the projection $P_r={\mathbf I}-(P_1+\ldots+P_{r-1})$, where $\mathbf I$ is the identity on $\mathcal H$. As a consequence of the spectrum assumption, we have $\sigma(\tilde E) \subset \{0\} \cup [\la_r,\infty)$, and $\la_r P_r \le \tilde E$. Hence, for all $i\in I$, $d_i \ge \sum_{j=1}^r \la_j p^{(j)}_i$, where $p_i^{(j)}=\langle P_j e_i, e_i\rangle$ are diagonal entries of $P_j$. Thus,
\[
\begin{aligned}
\sum_{d_i \le \la_r} (\la_r - d_i) \le \sum_{d_i \le \la_r}  \bigg( \la_r - \sum_{j=1}^r \la_j p^{(j)}_i\bigg)
&= \sum_{d_i \le \la_r}  \bigg( \la_r(1-p_i^{(r)}) - \sum_{j=1}^{r-1} \la_j p^{(j)}_i \bigg) 
\\
&= \sum_{d_i \le \la_r}  \sum_{j=1}^{r-1} (\la_r-\la_j) p^{(j)}_i.
\end{aligned}
\]
In the last step we used the fact that $p^{(1)}_i+\ldots +p^{(r)}_i=1$ for all $i\in I$. Since $\sum_{i\in I} p^{(j)}_i = N_j$ for $j=1,\ldots, r-1$, we have
\[
 \sum_{j=1}^{r-1} (\la_r-\la_j)N_j = \sum_{i\in I}  \sum_{j=1}^{r-1} (\la_r-\la_j) p^{(j)}_i
 \ge  \sum_{d_i \le \la_r}  \sum_{j=1}^{r-1} (\la_r-\la_j) p^{(j)}_i.
\]
Combining the last two estimates yields \eqref{ext1}.
\end{proof}

By the symmetry we automatically obtain a version of Theorem \ref{ext} for operators with discrete spectrum in the upper part.

\begin{thm}\label{exts}
Suppose that $E$ is a self-adjoint operator on $\mathcal H$ with 
\[
\sigma(E) \subset (-\infty,\la_m] \cup  \{\la_{m-1},\ldots,\la_1\},
\]
where $\la_m < \ldots < \la_1$, $m\ge 2$. Let $\{e_{i}\}_{i\in I}$ be an orthonormal basis for $\Hil$ and $d_{i}=\langle Ee_{i},e_{i}\rangle$. Then, for any $r=2,\ldots, m$,
\begin{equation}\label{exts1}
\sum_{i\in I,\ d_i \ge \la_r} (d_i - \la_r) \le \sum_{j=1}^{r-1} (\la_j-\la_r)N_j,
\qquad\text{where }N_j=m_E(\la_j).
\end{equation}
\end{thm}

\begin{proof}
Observe that $-E$ satisfies the hypothesis of Theorem \ref{ext} for the sequence $-\lambda_1< \ldots < -\lambda_m$. Then, Theorem \ref{ext} yields \eqref{exts1}.
\end{proof}

\begin{proof}[Proof of necessity of Theorem \ref{fullthm}(i)(iii)]
Suppose that $E$ is a self-adjoint operator with diagonal $\{d_{i}\}_{i\in I}$ and the eigenvalue-multiplicity list $\{(A_j,N_j)\}_{j=-m}^{n+p+1}$. Then, Theorems \ref{ext} and \ref{exts} yield (i) and (iii), resp.
\end{proof}

It is worth mentioning that the above results provide   majorization condition also for operators with an infinite discrete spectrum. Corollary \ref{exti} can be considered as an extension of majorization for trace class operators; compare with \cite{ak}.

\begin{cor}\label{exti}
Suppose that $\{\la_j\}_{j\in\N}$ is an decreasing sequence with limit $\la_\infty = \lim_{j\to\infty} \la_j$. Suppose that $E$ is a self-adjoint operator with 
\[
\sigma(E) \subset \{\la_1,\la_2, \ldots \} \cup (-\infty,\la_\infty].
\]
Then,
\begin{equation}\label{exti1}
\sum_{d_i \ge \la_\infty} (d_i - \la_\infty) \le \sum_{j=1}^{\infty} (\la_j-\la_\infty)N_j,
\qquad\text{where }N_j=m_E(\la_j).
\end{equation}
\end{cor}

\begin{proof}
Theorem \ref{exts} applies and yields inequality \eqref{exts1}. By letting $r\to \infty$ we obtain \eqref{exti1} by the monotone convergence theorem.
\end{proof}

\section{Necessity of interior majorization}\label{S4}

In this section we will show the necessity of the interior majorization in Theorem \ref{fullthm}. The first step in this two stage process is to establish the trace condition \eqref{fulltrace}. At the same time Theorem \ref{tt} shows that non-summability, i.e. $C(B/2)=\infty$ or $D(B/2)=\infty$, is the only option if more than two eigenvalues have infinite multiplicity.

\begin{thm}\label{tt} Let $E$ be a self-adjoint operator on $\mathcal H$ with the spectrum 
\[
\sigma(E)=\{A_{-m},\ldots,A_{n+p+1}\},
\]
where $m,n,p \in \N_0$ and $\{A_j\}_{j=-m}^{n+p+1}$ is an increasing sequence such that $A_0=0$ and $A_{n+1}=B$. Let $\{e_{i}\}_{i\in I}$ be an orthonormal basis for $\Hil$ and $d_{i}=\langle Ee_{i},e_{i}\rangle$. Assume that 
\begin{equation}\label{tt1}
N_{j}:=m_{E}(A_{j})<\infty \qquad\text{for all }j<0 \text{ and } j>n+1.
\end{equation}
Assume also that for some $0< \alpha<B$ both series 
\[
C(\alpha)=\sum_{d_{i}<\alpha}d_{i},\qquad D(\alpha)=\sum_{ d_i \ge \alpha}(B-d_{i})
\]
are absolutely convergent. Then, the following hold:
\begin{enumerate}
\item the series $C(\alpha)$ and $D(\alpha)$ are absolutely convergent for {\bf all} $0< \alpha<B$,
\item
the interior multiplicities are finite
\begin{equation}\label{tt3}
N_{j}=m_{E}(A_{j})<\infty \qquad\text{for all }  j=1,\ldots,n,
\end{equation}
\item
there exists $k=k(\alpha)\in\Z$ depending on $\alpha$ such that
\begin{equation}\label{tt4}
C(\alpha) -D(\alpha) =k B + \sum_{\genfrac{}{}{0pt}{}{j=-m}{j \not= 0,n+1}}^{n+p+1} N_{j}A_{j}.\end{equation}
\end{enumerate}

In addition, if we assume that $N_0=N_{n+1}=\infty$, then
\begin{equation}
\label{fullinf1}
|\{i\in I: d_i < \alpha \}| = |\{ i\in I: d_i \ge \alpha \}| = \infty.
\end{equation}
\end{thm}

\begin{proof}
By the spectral decomposition, we can write
\[
E= \sum_{j=-m}^{n+p+1} A_{j} P_j,
\]
where $P_j$'s are mutually orthogonal projections satisfying $\sum_{j=-m}^{n+p+1}P_j={\mathbf I}$. Let $p^{(j)}_{i}=\langle P_je_{i},e_{i}\rangle$ be the diagonal of $P_j$. Hence, we have
\begin{equation}\label{tt7}
\sum_{j=-m}^{n+p+1} p_i^{(j)} =1 \qquad\text{for all }i \in I.
\end{equation}

For convenience, we let
\begin{equation}\label{tt9}
\begin{aligned}
q_i^{(1)}&=\sum_{j=-m}^{-1} A_j p_i^{(j)} \le 0, \qquad
q_i^{(2)}=\sum_{j=n+2}^{n+p+1} A_j p_i^{(j)} \ge 0,
\\
\tilde d_i & = d_i - q_i^{(1)} - q_i^{(2)} = \sum_{j=1}^{n+1} A_j p_i^{(j)}\in [0,B].
\end{aligned}
\end{equation}
By \eqref{tt1}, the following two series are convergent
\begin{equation}\label{tt8}
\sum_{i\in I} q_i^{(1)} = \sum_{j=-m}^{-1} N_j A_j > -\infty,
\qquad
\sum_{i\in I} q_i^{(2)} = \sum_{j=n+2}^{n+p+1} N_j A_j <\infty.
\end{equation}
%Since $d_i \ge q_i^{(1)}$, the series $\sum_{d_{i}<0}d_{i}$ converges by \eqref{tt8}. Using $d_i - B = (\tilde d_i - B) + q_i^{(1)} + q_i^{(2)} \le  q_i^{(2)}$, the series 
%$\sum_{ d_i >B }(d_{i}-B)$ converges as well. Consequently, both series
%\begin{equation}\label{tt10}
%\sum_{d_{i}<\alpha}d_{i}=C+C_0,\qquad \sum_{\alpha\le d_i }(B-d_{i})=D+D_0
%\end{equation}
%are absolutely convergent.

For convenience we let $I_0=\{i\in I: d_{i}<\alpha\}$ and $I_1=\{i\in I: d_{i} \ge \alpha\}$. Since the series defining $C(\alpha)$ and $D(\alpha)$ are absolutely convergent, any interval $[\epsilon,B-\epsilon]$, $\epsilon>0$ may contain only finitely many $d_i$'s. Thus, $C(\alpha)$ and $D(\alpha)$ are absolutely convergent for all $0<\alpha<B$, which justifies (i). Moreover, by \eqref{tt9} and \eqref{tt8} the following two series are convergent
\begin{equation}\label{tt6}
\tilde C = \sum_{i\in I_0} \tilde d_i <\infty, \qquad \tilde D = \sum_{i\in I_1} (B- \tilde d_i)<\infty.
\end{equation}
By \eqref{tt9} and \eqref{tt6} we have
\begin{equation}\label{tt11}
\sum_{i\in I_0} p_i^{(j)} < \infty \qquad\text{for }j=1,\ldots,n+1.
\end{equation}
By \eqref{tt7} and \eqref{tt9} we have
\[
B - \tilde d_i 
= B\bigg(1 -  \sum_{j=1}^{n+1} p_i^{(j)}\bigg) +  \sum_{j=1}^{n+1} (B-A_j) p_i^{(j)} \ge
\sum_{j=1}^{n} (B-A_j) p_i^{(j)}.
\]
Summing the above inequality over $i\in I_1$, \eqref{tt6} yields
\begin{equation}\label{tt12} 
\sum_{i\in I_1} p_i^{(j)} < \infty \quad\text{for }j=1,\ldots,n.
\end{equation}
Thus, using \eqref{tt6} again and the identity $B-\tilde d_i = B(1-p_i^{(n+1)})-\sum_{j=1}^n A_j p_i^{(j)}$ we also have
\begin{equation}\label{tt12a}
\sum_{i\in I_1} (1-p_i^{(n+1)}) < \infty.
\end{equation}
Combining \eqref{tt11} and \eqref{tt12} proves (ii), i.e.,
\begin{equation}\label{tt13}
N_j = \sum_{i \in I} p_i^{(j)} < \infty \qquad j=1,\ldots,n.
\end{equation}

By \eqref{tt11} and \eqref{tt12a} and Remark \ref{rpart} we can apply Theorem \ref{Kadison} to the projection $P_{n+1}$ to deduce that
\[
k:=\sum_{i\in I_0} p_i^{(n+1)}-\sum_{i\in I_1} (1-p_i^{(n+1)}) \in \Z.
\]
Thus,
\[
\begin{aligned}
\tilde C - \tilde D &= \sum_{i\in I_0}  \bigg(Bp_i^{(n+1)} +  \sum_{j=1}^{n} A_j p_i^{(j)}  \bigg) - \sum_{i\in I_1} \bigg( B- Bp_i^{(n+1)}  - \sum_{j=1}^{n} A_j p_i^{(j)} \bigg) 
\\
&= Bk + \sum_{j=1}^n N_j A_j.
\end{aligned}
\]
Therefore, by \eqref{tt9} and \eqref{tt8} we have
\begin{equation}\label{tt15}
\begin{aligned}
&C(\alpha)-D(\alpha)
=\sum_{i\in I_0}d_{i}- \sum_{i \in I_1}(B-d_{i})
\\
&=
\sum_{i\in I_0}(\tilde d_{i} + q_i^{(1)} + q_i^{(2)}) - \sum_{i \in I_1}(B-\tilde d_{i} - q_i^{(1)} - q_i^{(2)})=Bk + \sum_{\genfrac{}{}{0pt}{}{j=-m}{j \not= 0,n+1}}^{n+p+1} N_{j}A_{j}.
\end{aligned}
\end{equation}
This shows \eqref{tt4} completing the proof of (i)--(iii) of Theorem \ref{tt}.

In addition, assume that $N_0=N_{n+1}=\infty$. It remains to show \eqref{fullinf1}. On the contrary, suppose that there are only finitely many $i\in I$ such that $d_i \ge \alpha$.
 Combining this with the assumption the series defining $C(\alpha)$ is absolutely convergent implies that $\sum_{i\in I} d_i$ is also absolutely convergent. By \eqref{tt9} and \eqref{tt8}, the series $\sum_{i\in I} \tilde d_i$ is convergent. Since $\tilde d_i = \sum_{j=1}^{n+1} A_j p_i^{(j)}$, by \eqref{tt13} the series $\sum_{i \in I} p^{(n+1)}_i$ is also convergent. This implies that the projection $P_{n+1}$ has finite rank which contradicts $N_{n+1}=\infty$. A similar argument shows that if there are only finitely many $i\in I$ such that $d_i < \alpha$, then the series $\sum_{i\in I} (B-d_i)$ converges absolutely, and hence $\sum_{i \in I} (1-p^{(n+1)}_i)$ is convergent. This implies that the rank of ${\mathbf I}-P_{n+1}$ is finite which contradicts the hypothesis that $N_0=\infty$.
\end{proof}

Once Theorem \ref{tt} is established it is now convenient to formalize the concept of interior majorization with the following definition.

\begin{defn}\label{mar}
Let $m,n,p \in \N_0$, $n\ge 1$, and let $\{A_j\}_{j=-m}^{n+p+1}$ be an increasing sequence such that $A_0=0$ and $A_{n+1}=B$. Let $\{N_j\}_{j=-m}^{n+p+1}$ be a sequence in $\N$ with the possible exception of $j=0$ and $j=n+1$, where $N_0$ and $N_{n+1}$ could also $=\infty$.
Let $\{d_i\}_{i\in I}$ be a sequence  in $\R$. Let $C(\alpha)$ and $D(\alpha)$ be as in Definition \ref{conv} with the convention that $C(\alpha)$ or $D(\alpha)=\infty$ if the corresponding series is not absolutely convergent. 

We say that $\{d_i\}$ satisfies {\it interior majorization} by $\{(A_j,N_j)\}_{j=-m}^{n+p+1}$ if the following 3 conditions hold:
\begin{enumerate}
\item
$C(B/2)<\infty$ and $D(B/2)<\infty$, and thus $C(\alpha)<\infty$ and $D(\alpha)< \infty$ for all $\alpha \in (0,B)$,
\item
there exists $k_0\in \Z$ such that
\begin{equation}
\label{fulltrace1} 
C(A_n) -D(A_n) =  \sum_{\genfrac{}{}{0pt}{}{j=-m}{j \not= 0,n+1}}^{n+p+1} A_j N_j +k_0 B,
\end{equation}
\item 
for all $r=1,\ldots,n$,
\begin{equation}
\label{fullmaj1} 
C(A_r) \ge\sum_{\genfrac{}{}{0pt}{}{j=-m}{j \not= 0}}^{r} A_j N_j + A_r \bigg( k_0 - |\{i\in I: A_r \le d_i < A_n \}|+\sum_{\genfrac{}{}{0pt}{}{j=r+1}{j\ne n+1}}^{n+p+1} N_j  \bigg).
\end{equation}
\end{enumerate}
\end{defn}

\begin{remark}
Despite initial appearance the interior majorization conditions \eqref{fulltrace1} and \eqref{fullmaj1} are equivalent with  \eqref{fulltrace} and \eqref{fullmaj} in Theorem \ref{fullthm}. Indeed, since the quantity $C(\alpha)-D(\alpha)$ remains constant modulo $B$ for all $\alpha \in (0,B)$, \eqref{fulltrace1} is equivalent to the statement that there exists $k=k(\alpha) \in \Z$ such that
\begin{equation}
\label{fulltrace2} 
C(\alpha) -D(\alpha) =  \sum_{\genfrac{}{}{0pt}{}{j=-m}{j \not= 0,n+1}}^{n+p+1} A_j N_j +k(\alpha) B,
\end{equation}
Fix $\alpha = A_r$, where $r=1, \ldots, n$. Since
$
k_0 - |\{i\in I: A_r \le d_i < A_n \}| = k(\alpha)
$,
\eqref{fullmaj1} can be rewritten as 
\begin{equation}
\label{fullmaj2} 
C(\alpha) \ge\sum_{\genfrac{}{}{0pt}{}{j=-m}{j \not= 0}}^{r} A_j N_j + \alpha \bigg( k(\alpha)+\sum_{\genfrac{}{}{0pt}{}{j=r+1}{j\ne n+1}}^{n+p+1} N_j  \bigg).
\end{equation}
Using \eqref{fulltrace2}, we can remove the presence of $k=k(\alpha)$ in \eqref{fullmaj2} to obtain
\begin{equation}
\label{fullmaj3} 
(B-\alpha)C(\alpha) +\alpha D(\alpha)  \ge
(B-\alpha)\sum_{\genfrac{}{}{0pt}{}{j=-m}{j \not= 0}}^{r} A_j N_j + \alpha \sum_{\genfrac{}{}{0pt}{}{j=r+1}{j \ne n+1}}^{n+p+1} (B-A_j)N_j.
\end{equation}
This is precisely \eqref{fullmaj} and the above process is reversible. Observe also that the value of $N_0$ and $N_{n+1}$ is irrelevant in Definition \ref{mar}. However, the most interesting case of interior majorization occurs when there are exactly two eigenvalues with infinite multiplicities $N_0=N_{n+1}=\infty$.
\end{remark}

We are now ready to establish the necessity of the interior majorization.

\begin{thm}\label{int-nec} Let $E$ be a self-adjoint operator on $\mathcal H$ satisfying the assumptions of Theorem \ref{tt}.  
Then, its diagonal
$\{d_i\}_{i\in I}$ satisfies interior majorization by $\{(A_j,N_j)\}_{j=-m}^{n+p+1}$.
\end{thm}

\begin{proof} We shall use the same notation as in the proof of Theorem \ref{tt}. The condition \eqref{tt3} guarantees that the sequence $\{N_j\}_{j=-m}^{n+p+1}$ fulfills the requirements of Definition \ref{mar}. Moreover, by letting $\alpha=A_n$, \eqref{tt4} yields the  trace condition \eqref{fulltrace1}. Thus, it remains to show the interior majorization inequality \eqref{fullmaj1}.

As in the proof of Theorem \ref{tt} we write
\[
E= \sum_{j=-m}^{n+p+1} A_{j} P_j.
\]
Since each orthogonal projection $P_j$ has rank $N_j$, by \eqref{tt1} and \eqref{tt3} we have
\begin{equation}\label{int2}
\sum_{i\in I}p^{(j)}_{i} = N_j<\infty
\qquad\text{for all } j = -m,\ldots, n+p+1, \ j \ne 0, n+1.
\end{equation}
For convenience we let $q_i=p_i^{(n+1)}$. Then, by \eqref{tt9} we have
\begin{equation}\label{int2a}
d_i =\sum_{\genfrac{}{}{0pt}{}{j=-m}{j \not= 0, n+1}}^{n+p+1} A_j  p_i^{(j)} + B q_i.
\end{equation}
Moreover, by \eqref{tt15} we have
\begin{equation}\label{int3}
\begin{aligned}
C(\alpha)-D(\alpha) = Bk(\alpha) + \sum_{\genfrac{}{}{0pt}{}{j=-m}{j \not= 0, n+1}}^{n+p+1}  A_{j} N_j,
\qquad\text{where }
k(\alpha)=\sum_{i\in I_0} q_i-\sum_{i\in I_1} (1-q_i) \in \Z.
 \end{aligned}
\end{equation}
Here, $I_0=\{i\in I: d_{i}<\alpha\}$ and $I_1=\{i\in I: d_{i} \ge \alpha\}$. By letting $\alpha=A_n$, we deduce that $k_0$ in \eqref{fulltrace1} must equal $k(A_n)$. 

Fix $r=1,\ldots,n$, and let $\alpha= A_r$. Then,
\[
\begin{aligned}
k_0
- |\{i\in I: A_r \le d_i < A_n \}|
& = \sum_{d_i < A_n} q_i - \sum_{d_i \ge A_n} (1-q_i) - |\{i\in I: A_r \le d_i < A_n \}|  
\\
& =\sum_{i\in I_0} q_i - \sum_{i\in I_1} (1-q_i) =k(A_r).
\end{aligned}
\]
Thus, by \eqref{int2a} the required majorization \eqref{fullmaj1} is equivalent to
\begin{equation}\label{int4}
\sum_{i\in I_0}\bigg( \sum_{\genfrac{}{}{0pt}{}{j=-m}{j \not= 0, n+1}}^{n+p+1} 
 A_j p_i^{(j)} +
 Bq_i \bigg) \ge \sum_{\genfrac{}{}{0pt}{}{j=-m}{j \ne 0}}^{r}A_{j} N_j + A_{r} \bigg(\sum_{i\in I_0} q_i - \sum_{i\in I_1} (1-q_i)  + \sum_{\genfrac{}{}{0pt}{}{j=r+1}{j \ne n+1}}^{n+p+1}N_{j}\bigg).
\end{equation}
By \eqref{int2}, we have for $j \ne 0,n+1$,
\[
N_j = \sum_{i\in I_0} p_i^{(j)} + \sum_{i\in I_1} p_i^{(j)}.
\]
Thus, \eqref{int4} can be rewritten as
\begin{equation}\label{int5}
\sum_{i\in I_0}\bigg( \sum_{\genfrac{}{}{0pt}{}{j=r+1}{j \not=n+1}}^{n+p+1} (A_j -A_r) p_i^{(j)} + (B-A_r)q_i \bigg) 
\ge \sum_{i\in I_1} \bigg( 
 \sum_{\genfrac{}{}{0pt}{}{j=-m}{j \ne 0}}^{r}A_{j} p_i^{(j)}+ 
 \sum_{\genfrac{}{}{0pt}{}{j=r+1}{j \ne n+1}}^{n+p+1} A_r p_i^{(j)}
 + A_{r} (q_i-1)  \bigg).
\end{equation}
Since $\{A_j\}$ is an increasing sequence, the left hand side of \eqref{int5} is $\ge 0$. On the other hand, the right hand side of \eqref{int5} is $\le 0$ as it is dominated by
\[
\sum_{i\in I_1} \bigg( 
 \sum_{\genfrac{}{}{0pt}{}{j=-m}{j \ne 0, n+1}}^{n+p+1}A_r p_i^{(j)}+ A_{r} q_i-A_r  \bigg) \le A_r \sum_{i\in I_1} \bigg( 
 \sum_{j=-m}^{n+p+1} p_i^{(j)} -1  \bigg) = 0.
\]
In the last step we used \eqref{tt7}. This shows \eqref{int5}, which implies \eqref{int4}, thus proving \eqref{fullmaj1}. This completes the proof of Theorem \ref{int-nec}.
\end{proof}

\begin{proof}[Proof of necessity of Theorem \ref{fullthm}(ii)]
Suppose that $E$ is a self-adjoint operator with diagonal $\{d_{i}\}_{i\in I}$ and the eigenvalue-multiplicity list $\{(A_j,N_j)\}_{j=-m}^{n+p+1}$ as in Definition \ref{conv}. Suppose first that $N_j$ takes the value of $\infty$ more than twice. That is, $N_j=\infty$ for some $1 \le j \le n$ in addition to $N_0=N_{n+1}=\infty$. Since \eqref{tt3} fails, by the contrapositive of Theorem \ref{tt}, we must necessarily have that
\[
\sum_{0<d_{i}<B/2}d_{i}=\infty,\qquad\text{or}\qquad \sum_{B/2\le d_i <B }(B-d_{i})=\infty.
\]
Thus, we have the non-summability scenario. 

Suppose next that $N_j$ takes the value of $\infty$ exactly twice. If the non-summability happens, then there is nothing to prove. Thus, it remains to consider the case when $C(B/2)<\infty$ and $D(B/2)<\infty$. Theorem \ref{tt} shows that \eqref{fullinf1} holds. In the case when $n\ge 1$, Theorem \ref{int-nec} shows $\{d_i\}_{i\in I}$ satisfies interior majorization by $\{(A_j,N_j)\}_{j=-m}^{n+p+1}$. Finally, in the case when $n=0$, Theorem \ref{tt} alone yields the required conclusion in Theorem \ref{fullthm}(ii).
\end{proof}

\section{Sufficiency of interior majorization}

The goal of this section is to show the sufficiency of the interior majorization in the case when the two outermost eigenvalues have infinite multiplicities. This corresponds to the case when $m=0$ and $p=0$ in Definition \ref{conv}, and thus exterior majorization is not present. To achieve this we shall introduce an alternative variant of interior majorization which works in the crucial case when $\{d_{i}\}$ can be indexed in nondecreasing order by $\Z$.

\begin{defn}\label{rim}
Suppose that $n\in \N$ and $\{A_{j}\}_{j=0}^{n+1}$ is an increasing sequence in $\R$ such that $A_0=0$ and $A_{n+1}=B$. Suppose $\{N_{j}\}_{j=1}^{n}$ is a sequence in $\N$ and $N_{0}=N_{j+1}=\infty$.  Define
\begin{equation}\label{rmaj0}\lambda_{i} = \begin{cases} 0 & i\leq 0\\ A_{r} & 1+\sum_{j=1}^{r-1}N_{j}\leq i\leq \sum_{j=1}^{r}N_{j}\\ B & i>\sum_{i=1}^{n}N_{j}.\end{cases}\end{equation}
Let $\{d_{i}\}_{i\in\Z}$ be a nondecreasing sequence in $[0,B]$ such that $\sum_{i=-\infty}^0 d_i <\infty$.
We say that $\{d_i\}$ satisfies {\it Riemann interior majorization} by $\{(A_{j},N_{j})\}_{j=0}^{n+1}$ if there exists $k\in\Z$ such that the following two hold
\begin{align}
\label{rmaj1}\delta_{m}:=\sum_{i=-\infty}^{m}(d_{i-k}-\lambda_{i}) &\geq  0\qquad \text{for all }m\in\Z,
\\
\label{rmaj2}\lim_{m\to\infty}\delta_{m} & = 0.
\end{align}
\end{defn}

To distinguish between two distinct types of interior majorization we shall frequently refer to the concept introduced in Definition \ref{mar} as Lebesgue interior majorization. This is done purposefully as an analogy between Riemann and Lebesgue integrals. %explain more
Theorem \ref{eqmajs} shows the equivalence of the concepts of Riemann and Lebesgue interior majorization for nondecreasing sequences.

\begin{thm}\label{eqmajs} Let $\{(A_j,N_j)\}_{j=0}^{n+1}$ be as in Definition \ref{rim}. Let $\{d_{i}\}_{i\in\Z}$ be a nondecreasing sequence in $[0,B]$.
Then, the sequence $\{d_{i}\}$ satisfies interior majorization by $\{(A_j,N_j)\}_{j=0}^{n+1}$ if and only if $\{d_{i}\}$ satisfies Riemann interior majorization by $\{(A_j,N_j)\}_{j=0}^{n+1}$.
\end{thm}

\begin{proof} Without loss of generality we may assume that $d_{i}<A_{1}$ $\iff$ $i\leq 0$. We will establish the following notation to be used in the proof. For $r=1,\ldots,n$, we set
\[
m_{r} = |\{i \in I: A_1 \le d_i<A_r\}|, \qquad \sigma_{r} = \sum_{j=1}^{r}N_{j}, \qquad \sigma_0=0.
\]
Note that for any $r=1,\ldots,n$, we have
\begin{align}
\label{rmaja} d_i < A_r & \iff  i \le m_r, \\
\label{rmajb} \la_i = A_r & \iff \sigma_{r-1} +1 \le i \le \sigma_r.
\end{align}
Therefore, we have
\begin{equation}\label{rmaj3}
C(A_r) = \sum_{i=-\infty}^{m_r} d_{i},\qquad 
D(A_r) = \sum_{i=m_{r}+1}^{\infty}(B-d_{i}).
\end{equation}

First, we assume that $\{d_{i}\}$ satisfies interior majorization as in Definition \ref{mar}. From the assumption that $C(B/2)<\infty$ we have $\sum_{i=-\infty}^{0}d_{i}=C(A_{1})<\infty$.

Let $k_{0}\in\Z$ be as in Definition \ref{mar}. We will show that \eqref{rmaj1} and \eqref{rmaj2} hold with $k=k_{0}+\sigma_{n}-m_{n}$. 

For $m>\sigma_n$
\begin{align}\label{rmaj4}
\delta_{m} & = \sum_{i=-\infty}^{m}(d_{i-k} - \lambda_{i}) 
= \sum_{i=-\infty}^{m}d_{i-k} - \sum_{i=1}^{m}\lambda_{i} = \sum_{i=-\infty}^{m-k}d_{i} -\sum_{j=1}^n A_j N_j - (m-\sigma_n)B.
\end{align}
Using \eqref{rmaj3}, for $m>m_n+k$
\begin{equation}\label{rmaj5}
\begin{aligned}
 \sum_{i=-\infty}^{m-k}d_{i} & = \sum_{i=-\infty}^{m_n} d_i - \sum_{i=m_n+1}^{m-k} (B-d_i-B) 
 \\
&= C(A_n) -D(A_n) + B(m-k-m_n) + \sum_{i=m-k+1}^\infty (B-d_i).
\end{aligned}
 \end{equation}
For $m>\max\{\sigma_{n},m_{n}+k\}$, combining  \eqref{fulltrace1}, \eqref{rmaj4}, and \eqref{rmaj5} yields
\[
\delta_m = k_0 B - (m-\sigma_n)B + B(m-k-m_n) +  \sum_{i=m-k+1}^\infty (B-d_i) = \sum_{i=m-k+1}^\infty (B-d_i).
\]
Since this series is convergent, we have $\delta_m  \to 0$ as $m\to \infty$. This establishes \eqref{rmaj2}.

To complete this direction of the proof we must show that $\delta_{m}\geq 0$ for all $m\in\Z$. Since $d_i \ge 0 = \la_i$ for $i\le 0$, we have $\delta_{m}\geq 0$ for $m\leq 0$. Moreover, since $d_{i-k}-\lambda_{i}\leq 0$ for all $i>\sigma_{n}$, and $\delta_{m}\to 0$ as $m\to\infty$, this implies $\delta_{m}\geq 0$ for all $m>\sigma_{n}$. Note that in the current notation, interior majorization inequality \eqref{fullmaj1} takes the following form
\begin{equation}\label{rmaj6} \sum_{i=-\infty}^{m_{r}+k}d_{i-k}\geq \sum_{i=1}^{\sigma_{r}}\lambda_{i} + (k+m_{r}-\sigma_{r})A_{r}
\qquad\text{for } r=1,\ldots,n.\end{equation}

We will prove by induction on $r=0,\ldots,n$ that $\delta_m \ge 0$ for $m=0,1,\ldots, \sigma_r$. The base case $r=0$ was shown above. 
Assume the inductive hypothesis is true for $r-1$, where $1\le r \le n$. We will show that that $\delta_{m}\geq 0$ for all $m = \sigma_{r-1}+1,\ldots,\sigma_{r}$. There are two cases to consider.

{\bf Case 1.} Assume that $m_{r}+k\leq \sigma_{r}$. First we will show that $\delta_{m_{r}+k}\geq 0$. If $m_{r}+k\leq \sigma_{r-1}$, then the inductive hypothesis implies that $\delta_{m_{r}+k}\geq 0$, so we may assume $\sigma_{r-1}+1\leq m_{r}+k \leq \sigma_{r}$. Using \eqref{rmaj6} and then \eqref{rmajb}
\begin{align*}
\delta_{m_{r}+k} & = \sum_{i=-\infty}^{m_{r}+k}(d_{i-k}-\lambda_{i}) = \sum_{i=-\infty}^{m_{r}+k}d_{i-k} - \sum_{i=1}^{m_{r}+k}\lambda_{i}
 \geq \sum_{i=1}^{\sigma_{r}}\lambda_{i} + (k+m_{r}-\sigma_{r})A_{r} - \sum_{i=1}^{m_{r}+k}\lambda_{i}\\
 & = (k+m_{r}-\sigma_{r})A_{r} + \sum_{i=m_{r}+k+1}^{\sigma_{r}}\lambda_{i}  = 0.
\end{align*}
By \eqref{rmaja} and \eqref{rmajb}
\[
\begin{aligned}
d_{i-k} - \la_i \ge A_r - \la_i  \ge 0 & \qquad\text{for } m_{r}+k\leq i\leq \sigma_{r},
\\
d_{i-k} - \la_i < A_r - \la_i =0 & \qquad\text{for } \sigma_{r-1}+1\leq i\leq m_{r}+k.
\end{aligned}
\]
Combining this with $\delta_{m_{r}+k}\geq 0$ implies that $\delta_{m}\geq 0$ for all $m = \sigma_{r-1}+1,\ldots,\sigma_{r}$.

{\bf Case 2.} Assume $m_{r}+k>\sigma_{r}$. Using \eqref{rmaj6} and then \eqref{rmaja}
\begin{align*}
\delta_{\sigma_{r}} & = \sum_{i=-\infty}^{\sigma_{r}}(d_{i-k}-\lambda_{i})
 = \sum_{i=-\infty}^{m_{r}+k}d_{i-k} - \sum_{i=\sigma_{r}+1}^{m_{r}+k}d_{i-k} - \sum_{i=1}^{\sigma_r} \lambda_i \\
 & \geq   \sum_{i=1}^{\sigma_{r}}\lambda_{i} + (k+m_{r}-\sigma_{r})A_{r} - \sum_{i=\sigma_{r}+1}^{m_{r}+k}d_{i-k} - \sum_{i=1}^{\sigma_{r}}\lambda_{i} > (k+m_{r}-\sigma_{r})A_{r} - \sum_{i=\sigma_{r}+1}^{m_{r}+k}A_{r} = 0.
\end{align*}
By \eqref{rmaja} and \eqref{rmajb}, $d_{i-k}-\lambda_{i} < A_r - \la_i=0$ for all $\sigma_{r-1}+1 \le i \le \sigma_{r}$
Combining this with $\delta_{\sigma_r}\geq 0$ implies that $\delta_{m}\geq 0$ for all $m = \sigma_{r-1}+1,\ldots,\sigma_{r}$. 
 This completes the inductive step and shows \eqref{rmaj1} holds. 

Conversely, assume that $\{d_i\}$ satisfies Riemann interior majorization. First, we note that \eqref{rmaj2} implies
\begin{equation}\label{a1}\begin{split}
0 = \sum_{i=-\infty}^{\infty}(d_{i-k}-\lambda_{i}) & = \sum_{i=-\infty}^{m_{n}+k}(d_{i-k}-\lambda_{i}) + \sum_{i=m_{n}+k+1}^{\infty}(d_{i-k}-\lambda_{i})\\
 & = \sum_{i=-\infty}^{m_{n}}d_{i} - \sum_{i=-\infty}^{m_{n}+k}\lambda_{i} - \sum_{i=m_{n}+1}^{\infty}(B-d_{i}) + \sum_{i=m_{n}+k+1}^{\infty}(B-\lambda_{i})\\
 & = C(A_n) - D(A_n) - \sum_{i=-\infty}^{m_{n}+k}\lambda_{i} + \sum_{i=m_{n}+k+1}^{\infty}(B-\lambda_{i}).
\end{split}\end{equation}
If $\sigma_{n}>m_{n}+k$, then we have
\[\sum_{i=-\infty}^{m_{n}+k}\lambda_{i} + \sum_{i=m_{n}+k+1}^{\infty}(\lambda_{i}-B) = \sum_{i=1}^{m_{n}+k} \la_i + \sum_{i=m_{n}+k+1}^{\sigma_{n}}(\lambda_{i}-B) = \sum_{i=1}^{\sigma_{n}}\lambda_{i} - (\sigma_{n}-m_{n}-k)B.\]
If $\sigma_{n}\leq m_{n}+k$, then we have
\[\sum_{i=-\infty}^{m_{n}+k}\lambda_{i} + \sum_{i=m_{n}+k+1}^{\infty}(\lambda_{i}-B) = \sum_{i=-\infty}^{\sigma_{n}}\lambda_{i} + \sum_{i=\sigma_{n}+1}^{m_{n}+k} B.\]
In either case, using \eqref{a1} we have
\[
C(A_n) - D(A_n) = \sum_{j=1}^{n}A_{j}N_{j} + (k+m_{n}-\sigma_{n})B.
\]
This shows that \eqref{fulltrace1} holds with $k_0=k+m_{n}-\sigma_{n}$. Finally, we must show that \eqref{rmaj6} holds for each $r=1,2,\ldots,n$.

Fix $r= 1,2,\ldots,n$ and assume $\sigma_{r}\geq m_{r}+k$. Using $\delta_{m_{r}+k}\geq 0$ and the fact that $\la_i \le A_r$ for $i \le \sigma_r$, we have
\begin{align*}
\sum_{i=-\infty}^{m_{r}+k}d_{i-k} & \geq \sum_{i=-\infty}^{m_{r}+k}\lambda_{i} = \sum_{i=-\infty}^{\sigma_{r}}\lambda_{i} - \sum_{i=m_{r}+k+1}^{\sigma_{r}}\lambda_{i} \geq \sum_{i=1}^{\sigma_{r}}\lambda_{i} - (\sigma_{r}-m_{r}-k)A_{r}.
\end{align*}
Next assume $\sigma_{r}<m_{r}+k$. Using $\delta_{m_{r}+k}\geq 0$ and the fact $\la_i \ge A_{r+1}>A_r$ for $i \ge \sigma_r+1$, we have
\[\sum_{i=-\infty}^{m_{r}+k}d_{i-k} \geq \sum_{i=-\infty}^{m_{r}+k}\lambda_{i} = \sum_{i=-\infty}^{\sigma_{r}}\lambda_{i} + \sum_{i=\sigma_{r}+1}^{m_{r}+k}\lambda_{i} \geq \sum_{i=-\infty}^{\sigma_{r}}\lambda_{i} + (m_{r}+k-\sigma_{r})A_r.\]
This proves that $\{d_i\}$ satisfies interior majorization as in Definition \ref{mar}. 
\end{proof}

The key result of this section is the sufficiency of Riemann interior majorization for the existence of a self-adjoint operators with prescribed eigenvalues and diagonal. For nondecreasing sequences ordered by $\Z$, the necessity of Riemann interior majorization follows immediately by what was shown in Section \ref{S4} and Theorem \ref{eqmajs}.

\begin{thm}\label{srim}
Let $\{(A_j,N_j)\}_{j=0}^{n+1}$ be as in Definition \ref{rim}. Let $\{d_{i}\}_{i\in\Z}$ be a nondecreasing sequence in $[0,B]$ which satisfies Riemann interior majorization by $\{(A_j,N_j)\}_{j=0}^{n+1}$.
Then, there is a self-adjoint operator $E$ with eigenvalue-multiplicity list $\{(A_j,N_j)\}_{j=0}^{n+1}$ and diagonal $\{d_{i}\}_{i\in\Z}$. That is, $\sigma(E) = \{0,A_{1},\ldots,A_{n},B\}$, and $m_{E}(A_{j}) = N_{j}$ for $j=0,\ldots,n+1$.
\end{thm}

\begin{proof} Set $\sigma=|\{i\in\Z:\lambda_{i}\neq 0,B\}|=\sum_{j=1}^{n}N_{j}$.
Without loss of generality we can assume that the sequence  $\{d_{i}\}_{i\in\Z}$ satisfies Riemann interior majorization in Definition \ref{rim} with $k=0$. This is because shifting a sequence does not affect the fact that it satisfies Riemann interior majorization.

The special case when there exists $i_0\in \Z$ such that
\begin{equation}\label{srimB}
d_i = B \qquad\text{for all } i>i_0
\end{equation}
follows directly from Theorem \ref{horn-ndec} applied to the sequences $\{\la_i\}_{i=1}^{M}$ and $\{d_i\}_{i=-\infty}^{M}$ for any $M\geq\max\{i_{0},\sigma\}$. Moreover, by symmetry considerations, as explained later in Case 3, one can also deal with the reciprocal case when
\begin{equation}\label{srim0}
d_i = 0 \qquad\text{for all } i<i_0.
\end{equation}
Thus, without loss of generality we can assume that neither \eqref{srimB} nor \eqref{srim0} holds. Since $\{d_{i}\}$ is nondecreasing this is equivalent to $d_{i}\in(0,B)$ for all $i\in\Z$. For convenience we note that for any $m\in\Z$ we have
\begin{equation}\label{srimC}
\delta_{m} = \sum_{i=-\infty}^{m}(d_{i}-\lambda_{i}) = \sum_{i=m+1}^{\infty}(\lambda_{i}-d_{i}) = \begin{cases} \sum_{i=-\infty}^{m}d_{i} & m\leq 0,
\\
\sum_{i=m+1}^{\infty}(B-d_{i}) & m\geq\sigma.
\end{cases}\end{equation}
Fix an integer $m_0 \in [0,\sigma]$ such that
\begin{equation}\label{min}\delta_{m_{0}} = \min\{\delta_{m}:0\leq m\leq\sigma\},\end{equation}
Obviously, $\delta_{m_{0}}\leq\min\{\delta_{0},\delta_{\sigma}\}$. The proof of Theorem \ref{srim} splits into three cases.

{\bf Case 1.} $\delta_{m_{0}}<\min\{\delta_{0},\delta_{\sigma}\}$. There are finite subsets $I_{0}\subset\Z\cap(-\infty,0]$ and $I_{1}\subset \Z\cap[\sigma+1,\infty)$ such that
\[\sum_{i\in I_{0}}d_{i}>\delta_{m_{0}}\quad\text{and}\quad \sum_{i\in I_{1}}(B-d_{i})>\delta_{m_{0}}.
\]
We apply Lemma \ref{movelemma} (i) to $\{d_{i}\}$ on the interval $[0,B]$ with $\eta_{0} = \delta_{m_{0}}$, to obtain $\{\tilde{d}_{i}\}_{i\in\Z}$. 

We will show that $\{\lambda_{i}\}_{i=1}^{m_{0}}$ and $\{\tilde{d}_{i}\}_{i=-\infty}^{m_{0}}$ satisfy the conditions of Theorem \ref{horn-ndec}. From \eqref{ops1} and the assumption that $\{d_{i}\}$ is nondecreasing we see that $\tilde{d}_{i}\leq d_{i}\leq d_{1} = \tilde{d}_{1}$ for all $i\leq 1$. Since $\{d_{i}\}_{i\in\Z}$ is nondecreasing and $d_{i}=\tilde{d}_{i}$ for $i=1,\ldots,m_{0}$ we see that $\{\tilde{d}_{i}\}_{i=1}^{m_{0}}$ is nondecreasing. For $1\leq m\leq m_{0}$ we have
\[
\sum_{i=-\infty}^{m} \tilde{d}_{i} -  \sum_{i=1}^{m}\lambda_{i}
=
\sum_{i=-\infty}^{0} \tilde{d}_{i} +  \sum_{i=1}^{m}(\tilde{d}_{i}-\lambda_{i}) = \sum_{i=-\infty}^{0} d_{i} - \eta_0 +  \sum_{i=1}^{m}( d_{i}-\lambda_{i})
= \delta_m -\delta_{m_0} \ge 0,
\]
with equality when $m=m_0$. By Theorem \ref{horn-ndec} there is a positive operator $\tilde{E}_{0}$ with positive eigenvalues $\{\lambda_{i}\}_{i=1}^{m_{0}}$ and diagonal $\{\tilde{d}_{i}\}_{i=-\infty}^{m_{0}}$. Since the diagonal of $\tilde{E}_{0}$ is an infinite sequence we also have $m_{\tilde{E}_{0}}(0)=\infty$.

By a similar argument applied to the sequence $\{B-\tilde{d}_{i}\}_{i=m_{0}+1}^{\infty}$ and an appeal to Theorem \ref{horn-ninc} there is a positive operator $\tilde{E}_{B}$ with positive eigenvalues $\{B-\lambda_{i}\}_{i=m_{0}+1}^{\sigma}$, diagonal $\{B-\tilde{d}_{i}\}_{i=m_{0}+1}^{\infty}$ and $m_{\tilde{E}_{B}}(0) = \infty$.

Let $\tilde{E} = \tilde{E}_{0}\oplus(B\mathbf{I}-\tilde{E}_{B})$. Note that $\tilde{E}$ has eigenvalue-multiplicity list $\{(A_{j},N_{j})\}_{j=0}^{n+1}$ and diagonal $\{\tilde{d}_{i}\}_{i\in\Z}$. By Theorem \ref{movelemma} (ii) there is an operator $E$, unitarily equivalent to $E$, with diagonal $\{d_{i}\}_{i\in\Z}$. This completes the first case.

{\bf Case 2.} $\delta_{m_{0}} = \delta_{\sigma} \le \delta_{0}$. 
The proof of Case 2 breaks into two subcases. In subcase (i) we assume that there is a (finite or infinite) set $I_0\subseteq\Z\cap(-\infty,0]$ such that
\begin{equation}\label{maj3}
\sum_{i\in I_0}d_{i} = \sum_{i=1}^{\sigma}(\lambda_{i}-d_{i}).
\end{equation}
In subcase (ii) we assume that there exists a {\bf finite} set $I_0\subseteq\Z\cap(-\infty,0]$ such that
\begin{equation}\label{maj4}
\sum_{i\in I_0}d_{i} > \sum_{i=1}^{\sigma}(\lambda_{i}-d_{i}).
\end{equation}
Observe that
\begin{equation}\label{maj2}
\sum_{i=m+1}^{\sigma}(\lambda_{i}-d_{i}) = \delta_{m}-\delta_{\sigma}\geq 0\qquad m=0,1,\ldots,\sigma,
\end{equation}
which implies that
\begin{equation}\label{maj2a}
\sum_{i=-\infty}^{0}d_{i} = \delta_{0} \ge\delta_{0} - \delta_{\sigma} = \sum_{i=1}^{\sigma}(\lambda_{i}-d_{i}) \ge 0.
\end{equation}
From \eqref{maj2a} we see that if subcase (ii) fails, then we must have $\sum_{i=-\infty}^{0}d_{i} = \sum_{i=1}^{\sigma}(d_{i}-\lambda_{i})$ and we are in subcase (i).

First, assume we are in subcase (i). If $I_0$ is finite, then $\{d_{i}\}_{i\in I_0 \cup\{1,\ldots,\sigma\}}$ and the sequence $\{\lambda_{i}\}_{i\in I_0 \cup\{1,\ldots,\sigma\}}$, consisting of $|I_0|$ zeros and  $\{\lambda_{i}\}_{i=1}^{\sigma}$, satisfy majorization property of the Schur-Horn Theorem \ref{horn} (after reversing indexing). If $I_0$ is infinite, then the assumption that $\{d_{i}\}$ is nondecreasing guarantees that the assumptions of Theorem \ref{horn-ndec} are also met. The fact that $d_i$'s for $i\le 0$ are indexed by $I_0$ does not cause any problem here since one can temporarily reindex $\{d_{i}\}_{i\in I_0}$ into $\{d_i\}_{i=-\infty}^0$. 
Therefore, either Theorem \ref{horn} or Theorem \ref{horn-ndec} implies that there is a positive rank $\sigma$ operator $E_{0}$ with diagonal $\{d_{i}\}_{i\in I_0 \cup\{1,\ldots,\sigma\}}$ and spectrum $\sigma(E_0)=\{0,A_1,\ldots,A_n\}$, $m_{E_{0}}(A_{j})=N_{j}$ for each $j=1,\ldots,n$ and $m_{0}(E_{0})=|I_{0}|$. We shall establish that a similar conclusion holds in subcase (ii), albeit with appropriately modified diagonal terms.

Next, we assume we are in subcase (ii). Set
\[
\eta_0 := \sum_{i\in I_0}d_{i} - \sum_{i=1}^{\sigma}(\lambda_{i}-d_{i})>0.
\]
Observe that
\[
\delta_{\sigma} = \sum_{i=\sigma+1}^{\infty}(B-d_{i}) = \sum_{i=-\infty}^{\sigma}(d_{i} - \lambda_{i}) = \sum_{i=-\infty}^{0}d_{i} - \sum_{i=1}^{\sigma}(\lambda_{i}-d_{i})> \sum_{i\in I_0}d_{i} - \sum_{i=1}^{\sigma}(\lambda_{i}-d_{i}) = \eta_0.
\]
The strict inequality above is a consequence of our assumption that \eqref{srim0} fails.
Hence, there is a finite set $I_1\subset\Z\cap[\sigma+1,\infty)$ such that
\[
\sum_{i\in I_1}(B-d_{i}) > \eta_0.
\]
We apply Lemma \ref{movelemma} (i) to the sequence $\{d_{i}\}_{i\in \Z}$ on the interval $[0,B]$ with $\eta_{0}$ to obtain sequence $\{\tilde{d}_{i}\}_{i\in \Z}$. 
In particular, we have
\begin{align}\label{rsim10}
\sum_{i\in I_0}\tilde{d}_{i} &= \sum_{i\in I_0}d_{i} - \eta_{0} = \sum_{i=1}^{\sigma}(\lambda_{i}-d_{i}),
\\
\label{rsim11}
\sum_{i= \sigma+1}^\infty (B- \tilde d_i) & = \sum_{i= \sigma+1}^\infty (B- d_i) - \eta_0=\sum_{i=1}^\infty (\lambda_i-d_i)-\sum_{i\in I_0} d_i.
\end{align}

Combining the fact that $d_{i} = \tilde{d}_{i}$ for $i=1,2,\ldots,\sigma$ with \eqref{rsim10} yields 
\[
\sum_{i\in I_0}\tilde{d}_{i} + \sum_{i=1}^{m}(\tilde{d}_{i}-\lambda_{i}) 
= \sum_{i\in I_0 }\tilde{d}_{i} + \sum_{i=1}^{m}(d_{i}-\lambda_{i}) = \sum_{i=m+1}^{\sigma}(\lambda_{i}-d_{i})
 \geq 0  \qquad\text{for } m=1,2,\ldots,\sigma,
\]
with equality when $m=\sigma$. Since $\tilde{d}_{i}\leq d_{i}\leq d_{1}$ for all $i\in I_0$ this shows that the sequence $\{\tilde{d}_{i}\}_{i\in I_0 \cup\{1,\ldots,\sigma\}}$ and the sequence $\{\lambda_{i}\}_{i\in I_0 \cup\{1,\ldots,\sigma\}}$, consisting of $|I_0|$ zeros and  $\{\lambda_{i}\}_{i=1}^{\sigma}$, satisfy majorization property of the Schur-Horn Theorem \ref{horn} (with reverse ordering). Thus, there exists an operator $\tilde{E}_{0}$ with diagonal $\{\tilde{d}_{i}\}_{i\in I_0\cup\{1,\ldots,\sigma\}}$ and $\sigma(\tilde E_{0}) = \{0,A_{1},\ldots,A_{n}\}$, $m_{\tilde{E}_{0}}(A_{j})=N_{j}$ for $j=1,\ldots,n$, and $m_{\tilde{E}_{0}}(0)=|I_{0}|$. This was also shown in subcase (i), albeit with $\tilde d_i=d_i$ and $\tilde E_0 = E_0$. One can think of a trivial application of Lemma \ref{movelemma} (i) with $\eta_0=0$ in subcase (i). Thus, both subcases yield the same conclusion.

To finish the proof of Case 2 we set $I'_{0} = (\Z\cap(-\infty,0])\setminus  I_0$. By \eqref{rsim11} and \eqref{srimC} we have
\[
\sum_{i\in I'_0}\tilde{d}_{i} - \sum_{i=\sigma+1}^{\infty}(B-\tilde{d}_{i}) = \sum_{i\in I_0 \cup I'_0} d_{i} - \sum_{i=1}^{\infty}(\lambda_i-d_{i}) = 0.
\]
By Theorem \ref{Kadison} there is a projection $\tilde{P}$ such that $B\tilde{P}$ has diagonal $\{\tilde{d}_{i}\}_{i\in I'_0 \cup\{\sigma+1,\sigma+2,\ldots\}}$. Since $\sum_{i>\sigma}\tilde{d}_{i}=\infty$ it is clear that $m_{B\tilde{P}}(B)=\infty$. If $|I_{0}|<\infty$, then $\sum_{i\in I'_{0}}(B-\tilde{d}_{i})=\infty$ and thus $m_{B\tilde{P}}(0)=\infty$. Consequently, $\tilde E = B\tilde{P}\oplus \tilde{E}_{0}$ has eigenvalue-multiplicity list $\{(A_{j},N_{j}\}_{j=0}^{n+1}$, and diagonal $\{\tilde{d}_{i}\}_{i\in\Z}$. By Lemma \ref{movelemma} (ii) there is an operator $E$ which is unitarily equivalent to $\tilde E$ with diagonal $\{d_{i}\}_{i\in\Z}$. This completes the proof of Case 2.

{\bf Case 3.} $\delta_{m_{0}}=\delta_{0} \le \delta_{\sigma}$. Define the sequences $d'_i=B-d_{-i+\sigma+1}$ and $\lambda'_i=B-\lambda_{-i+\sigma+1}$ for all $i\in\Z$. One can verify that $\{d'_i\}$ satisfies Riemann interior majorization with respect to the sequence $\{(B-A_{-j+n+1},N_{-j+n+1})\}_{j=0}^{n+1}$. With this modification the sequence $\{d'_i\}$ satisfies the requirements of Case 2. Hence, there exists an operator $E'$ with eigenvalue list $\{\lambda'_{i}\}$ and diagonal $\{d'_{i}\}$. Therefore, the operator $E:=B\mathbf{I}-E'$ has the desired properties. This completes the proof of Case 3 and the theorem.
\end{proof}

We will frequently find it useful to append zeros and $B$'s to a sequence in order to be able to apply Theorem \ref{srim} to construct an operator. The following lemma shows how these appended diagonal terms may be removed from an operator.

\begin{lem}\label{decomp} Let $E$ be a positive operator on a Hilbert space $\Hil$ with with $\norm{E} \leq B$. Let $\{e_{i}\}_{i\in I}$ be an orthonormal basis for $\Hil$ and set $d_{i} = \langle Ee_{i},e_{i}\rangle$ for $i\in I$. Set $\Kil=\overline{\lspan}\{e_{i} :d_{i}\in(0,B)\}$ and $\Hil_{\lambda} =\overline{\lspan}\{e_{i} :d_{i} = \lambda\}$ for $\lambda = 0,B$. There exists a positive operator $E_{0}:\Kil\to\Kil$ such that $E = \mathbf{0}_{0}\oplus E_{0}\oplus B\mathbf{I}_{B}$ where $\mathbf{0}_{0}$ is the zero operator on $\Hil_{0}$ and $\mathbf{I}_{B}$ is the identity operator $\Hil_{B}$. In particular, $E_{0}$ has diagonal $\{d_{i}\}_{d_{i}\in(0,B)}$, $\sigma(E_{0})\subset\sigma(E)$, and $m_{E_{0}}(\lambda) = m_{E}(\lambda)$ for all $\lambda\in(0,B)$.
\end{lem}

\begin{proof} Let $i\in I$ such that $d_{i} = B$. Using the Cauchy-Schwarz inequality we have
\[B = \langle Ee_{i},e_{i}\rangle \leq \norm{Ee_{i}}\leq \norm{E}\leq B.\]
This implies
\[\langle Ee_{i},e_{i}\rangle = \norm{Ee_{i}}= \norm{E}=B.\]
Equality in the Cauchy-Schwarz inequality implies that the vectors $Ee_{i}$ and $e_{i}$ are linearly dependent. Since $E\geq0$ there is some $\lambda\in[0,\infty)$ such that $Ee_{i} = \lambda e_{i}$. Next, we see
\[B = \norm{Ee_{i}} = \norm{\lambda e_{i}} = \lambda.\]
Thus $e_{i}$ is an eigenvector with eigenvalue $B$. Since $\{e_{i}:d_{i}=B\}$ is a basis for $\Hil_{B}$ this shows that every nonzero $f\in\Hil_{B}$ is an eigenvector with eigenvalue $B$.
Applying the previous argument to the operator $B-E$ we see that $\Hil_{0}\subset\ker(E)$.

Next, we claim that $\mathcal{K}$ is invariant under $E$. If $f\in\Kil$, $g\in\Hil_{0}$, and $h\in\Hil_{B}$ then
\[\langle Ef,Eg\rangle = \langle Ef,0\rangle = 0 = B^{2}\langle f,h\rangle = \langle f,B^{2}h\rangle = \langle f,E^{2}h\rangle = \langle Ef,Eh\rangle,\]
which implies $Ef\in\Hil_{0}^{\perp}$ and $Ef\in\Hil_{B}^{\perp}$, that is $Ef\in\Kil$. Finally, define $E_{0}:\Kil\to\Kil$ by $E_{0}f = Ef$ for $f\in\Kil$.\end{proof}

Finally, combining Theorems \ref{eqmajs} and \ref{srim} we can show the sufficiency of Lebesgue interior majorization when exterior majorization is not present. In essence, Theorem \ref{inthorn} deals with sequences which satisfy Lebesgue interior majorization, but do not conform to more restrictive Riemann interior majorization. We wish to emphasize that the index set $I$ below can be either finite or (countably) infinite. In the short run this forces us to consider additional cases. In the long run Theorem \ref{inthorn} will enable us to streamline the proof of the sufficiency direction in Theorem \ref{fullthm}.

\begin{thm}\label{inthorn}
Let $\{(A_j,N_j)\}_{j=0}^{n+1}$ be as in Definition \ref{rim}. Let $\{d_{i}\}_{i\in I}$ be a sequence in $[0,B]$ which satisfies interior majorization by  $\{(A_j,N_j)\}_{j=0}^{n+1}$.
Then, there is a self-adjoint operator $E$ with spectrum $\sigma(E) \subset \{0,A_1,\ldots, A_n, B\}$ such that $m_E(A_j)=N_j$ for $j=1,\ldots,n$. In addition, if $s_1=|\{i \in I:d_i<A_1\}|=\infty$, then $m_E(0)=\infty$. Likewise, if $s_{n}=|\{i \in I:d_i>A_n\}|=\infty$, then $m_E(B)=\infty$.
\end{thm}

\begin{proof} Set $J:=\{i\in I:d_{i}\in(0,B)\}$ and $J_{\lambda}:=\{i:d_{i}=\lambda\}$ for $\lambda=0,B$. Let ${\mathbf I}$ be the identity operator on a space of dimension $|J_{B}|$ and let $\bf{0}$ be the zero operator on a space of dimension $|J_{0}|$. Since $C(B/2)<\infty$ and $D(B/2)<\infty$, the only possible limit points of $\{d_{i}\}_{i\in J}$ are $0$ and $B$. The argument breaks into four cases depending on the number of limit points.

{\bf Case 1:} Assume both $0$ and $B$ are limit points of the sequence $\{d_{i}\}_{i\in J}$. Note that in this case $s_{1}=s_{n}=\infty$. This implies that there is a bijection $\pi:\Z\to J$ such that $\{d_{\pi(i)}\}_{i\in\Z}$ is in nondecreasing order. Since $\{d_{i}\}_{i\in J}$ still satisfies Lebesgue interior majorization, by Theorem \ref{eqmajs} the sequence $\{d_{\pi(i)}\}_{i\in \Z}$ satisfies Riemann interior majorization. By Theorem \ref{srim} there is a positive operator $E'$ with diagonal $\{d_{i}\}_{i\in J}$ and eigenvalue-multiplicity list $\{(A_j,N_j)\}_{j=0}^{n+1}$. The operator $E'\oplus B {\mathbf I}\oplus \bf{0}$ is as desired. This completes the proof of Case 1.

{\bf Case 2:} Assume $0$ is the only limit point of $\{d_{i}\}_{i\in J}$. Note that $s_{1}=\infty$ in this case. There is a bijection $\pi:-\N\to J$ such that $\{d_{\pi(i)}\}_{i=-\infty}^{-1}$ is in nondecreasing order. Define the sequence $\{d_{i}'\}_{i\in\Z}$ by $d_{i}'=d_{\pi(i)}$ for $i<0$ and $d_{i}'=B$ for $i\geq 0$. The sequence $\{d_{i}'\}$ satisfies Lebesgue interior majorization. Theorem \ref{eqmajs} implies that $\{d_{i}'\}$ also satisfies Riemann interior majorization. By Theorem \ref{srim} there is a positive operator $E'$ with diagonal $\{d_{i}'\}_{i\in \Z}$ and eigenvalue-multiplicity list $\{(A_j,N_j)\}_{j=0}^{n+1}$.

By Lemma \ref{decomp} there is a positive operator $E_{0}$ with diagonal $\{d_{i}\}_{i\in J}$, spectrum $\sigma(E_{0})\subset\{A_{0},\ldots,A_{n+1}\}$, and $m_{E_{0}}(A_{j}) = N_{j}$ for $j=1,\ldots,n$. Moreover, since $\{d_{i}\}_{i\in J}$ is an infinite summable sequence, we see that $E_{0}$ is finite rank and $m_{E_{0}}(0) = \infty$. The operator $E=E_{0}\oplus {\mathbf 0} \oplus B{\mathbf I}$ has diagonal $\{d_{i}\}_{i\in I}$ and $m_{E}(A_{j}) = N_{j}$ for $j=1,\ldots,n$. In addition, if $s_{n}=\infty$, then $J_{B}$ must be infinite, and hence $m_{E}(B) = \infty$. This completes the proof of Case 2.

{\bf Case 3:} Assume $B$ is the only limit point of $\{d_{i}\}_{i\in J}$. The proof of this case follows by an obvious modification of Case 2.

{\bf Case 4:} Assume $\{d_{i}\}_{i\in J}$ has no limit points. This implies $J$ is a finite set. There is a bijection $\pi:\{1,2,\ldots,|J|\}\to J$ such that $\{d_{\pi(i)}\}_{i=1}^{|J|}$ is nondecreasing. Define the sequence $\{d_{i}'\}_{i\in\Z}$ by 
\[
d_{i}'= \begin{cases} 0 &  i\leq0,
\\
d_{\pi(i)} & i=1,\ldots,|J|,
\\
B & i>|J|.
\end{cases}
\]
Note that $\{d_{i}'\}$ satisfies interior majorization, and Theorem \ref{eqmajs} implies that it also satisfies Riemann interior majorization. By Theorem \ref{srim} there is a self-adjoint operator $E'$ with diagonal $\{d_{i}'\}_{i\in \Z}$ and eigenvalue-multiplicity list $\{(A_j,N_j)\}_{j=0}^{n+1}$.

By Lemma \ref{decomp} there is a positive operator $E_{0}$ with diagonal $\{d_{i}\}_{i\in J}$, spectrum $\sigma(E_{0})\subset\{A_{0},\ldots,A_{n+1}\}$, and $m_{E_{0}}(A_{j}) = N_{j}$ for $j=1,\ldots,n$. The operator $E = \mathbf{0}\oplus E_{0}\oplus B\mathbf{I}$ has diagonal $\{d_{i}\}_{i\in I}$ and $m_{E}(A_{j}) = N_{j}$ for $j=1,\ldots,n$. In addition, since $J$ is a finite set, $s_{n}=\infty$ implies that $J_{B}$ is infinite, and hence $m_{E}(B) = \infty$. Similarly, if $s_{1}=\infty$, then $J_{0}$ is infinite, and hence $m_{E}(0)=\infty$. This completes the proof of Case 4 and the theorem.

\end{proof}

\section{Sufficiency of exterior majorization}

The goal of this section is to show the sufficiency of the exterior majorization in the case when there is exactly one eigenvalue with infinite multiplicity. This corresponds to the case when interior majorization is not present and yields the Schur-Horn theorem for finite rank (not necessarily positive) self-adjoint operators.

\begin{defn}\label{ulmaj}
Let $\{A_{j}\}_{j=0}^{p}$ be an increasing sequence. For each $j=1,\ldots,p$, let $N_{j}\in\N$ and let $N_{0}\in\N\cup\{\infty\}$. We say that $\{d_{i}\}$ satisfies {\it upper exterior majorization} by $\{(A_{j},N_{j})\}_{j=0}^{p}$ if for each $r=0,1,\ldots,p$
\begin{equation}\label{ulmaj1}
\sum_{d_{i}\geq A_{r}}(d_{i}-A_{r}) \leq \sum_{j=r+1}^{p}N_{j}(A_{j}-A_{r})<\infty.
\end{equation}
\end{defn}

\begin{remark}
Note that the value of $N_{0}$ does not play any role in Definition \ref{ulmaj}. Nevertheless, it is convenient to include $N_0$ in the above definition in order to form the pair $(A_0,N_0)$.
\end{remark}

\begin{thm}\label{pfrhorn} Let $\{A_{j}\}_{j=0}^{p}$ be an increasing sequence, let $N_{j}\in\N$ for each $j=1,2,\ldots,p$, and let $N_0\in\N\cup\{\infty\}$. If a sequence $\{d_{i}\}_{i\in I}$ in $[A_{0},A_{p}]$ satisfies upper exterior majorization by $\{(A_{j},N_{j})\}_{j=0}^{p}$ and
\begin{equation}\label{pfrhorn1}
\sum_{i\in I}(d_{i}-A_0) = \sum_{j=1}^{p}N_{j}(A_{j}-A_0),\end{equation}
then there exists a positive operator $E$ with diagonal $\{d_{i}\}_{i\in I}$ and the following properties:
\begin{align}\label{pfrhorn2}\sigma(E) &\subseteq \{A_{0},\ldots,A_{p}\},
\\
\label{pfrhorn3} m_{E}(A_{j}) & = N_{j} \text{ for each }j=1,\ldots,p,
\\
\label{pfrhorn4} m_{E}(A_{0}) & = |I|-\sum_{j=1}^{p}N_{j}\geq 0.
\end{align}
\end{thm}

\begin{remark}
Note that the number of terms in the diagonal sequence is not assumed to be equal to the sum of the multiplicities of the eigenvalues. Indeed, upper exterior majorization combined with the trace condition \eqref{pfrhorn1} guarantees that the number of terms is at least the sum of the multiplicities of the positive eigenvalues.
\end{remark}

\begin{proof} Without loss of generality we can assume that $A_0=0$. Fix some $B>A_{p}$ and set $A_{p+1}=B$ and $N_{p+1} = \infty$. We claim that $\{d_{i}\}_{i\in I}$ satisfies interior majorization by $\{(A_{j},N_{j})\}_{j=0}^{p+1}$ with $k_{0} = -|\{i\in I:d_{i} = A_{p}\}|$. From \eqref{pfrhorn1} we have
\[C(A_{p}) - D(A_{p}) = \sum_{d_{i}<A_{p}}d_{i} - \sum_{d_{i}\geq A_{p}}(B-d_{i}) = \sum_{i\in I}d_{i} - B|\{i\in I:d_{i} = A_{p}\}| = \sum_{j=1}^{p}A_{j}N_{j} + Bk_{0},\]
which shows \eqref{fulltrace1}. Using \eqref{ulmaj1}, for $r=1,\ldots,p$ we have
\begin{align*}
C(A_{r}) &  = \sum_{d_{i}<A_{r}}d_{i} = \sum_{j=1}^{p}A_{j}N_{j} - \sum_{d_{i}\geq A_{r}}d_{i}\\
 & = \sum_{j=1}^{p}A_{j}N_{j} - A_{r}|\{i\in I:A_{r}\leq d_{i}\leq A_{p}\}|  - \sum_{d_{i}\geq A_{r}}(d_{i}-A_{r})\\
 & \geq \sum_{j=1}^{p}A_{j}N_{j} - A_{r}|\{i\in I:A_{r}\leq d_{i}\leq A_{p}\}|  - \sum_{j=r+1}^{p}N_{j}(A_{j}-A_{r})\\
 & = \sum_{j=1}^{r}A_{j}N_{j} + A_{r}\left(-|\{i\in I:A_{r}\leq d_{i}\leq A_{p}\}|  + \sum_{j=r+1}^{p}N_{j}\right)\\
 & = \sum_{j=1}^{r}A_{j}N_{j} + A_{r}\left(k_{0}-|\{i\in I:A_{r}\leq d_{i}< A_{p}\}|  + \sum_{j=r+1}^{p}N_{j}\right).
\end{align*}
The above calculation shows \eqref{fullmaj1} and proves the claim that $\{d_{i}\}_{i\in I}$ satisfies interior majorization by $\{(A_{j},N_{j})\}_{j=0}^{p+1}$.

By Theorem \ref{inthorn}, there is a positive operator $E$ with diagonal $\{d_{i}\}_{i\in I}$, spectrum $\sigma(E)\subset\{A_{0},\ldots,A_{p+1}\}$ and \eqref{pfrhorn3}. The operator $E$ has a finite spectrum and summable diagonal. This implies that $E$ is a finite rank operator. From \eqref{pfrhorn} and the fact that $m_{E}(A_{j})=N_{j}$ for $j=1,\ldots,p$, we conclude that $A_{p+1}\notin\sigma(E)$ and thus we have \eqref{pfrhorn2}. Finally, \eqref{pfrhorn4} follows from \eqref{pfrhorn2}, and \eqref{pfrhorn3}.
\end{proof}

\begin{defn}\label{llmaj}
Let $\{A_{j}\}_{j=-m}^{0}$ be an increasing sequence. For each $j=-m,\ldots,-1$ let $N_{j}\in\N$ and let $N_{0}\in\N\cup\{\infty\}$. We say that a sequence $\{d_{i}\}_{i\in I}$ satisfies {\it lower exterior majorization} by $\{(A_{j},N_{j})\}_{j=-m}^{0}$ if for each $r=-m,-m+1,\ldots,0$
\begin{equation}\label{llmaj1}
\sum_{d_{i}\leq A_{r}}(A_{r}-d_{i}) \leq \sum_{j=-m}^{r-1}N_{j}(A_{r}-A_{j}).
\end{equation}
\end{defn}

\begin{remark}\label{r1}
A sequence $\{d_{i}\}_{i\in I}$ satisfies a {\it lower exterior majorization} by $\{(A_{j},N_{j})\}_{j=-m}^{0}$ if and only if  $\{-d_{i}\}_{i\in I}$ satisfies an upper exterior majorization by $\{(-A_{-j}, N_{-j})\}_{j=0}^{m}$.
\end{remark}

In light of Remark \ref{r1}, the following version of Horn's Theorem for negative finite rank operators follows from Theorem \ref{pfrhorn}.

\begin{thm}\label{nfrhorn} Let $\{A_{j}\}_{j=-m}^{0}$ be an increasing sequence with $A_{0}=0$, let $N_{j}\in\N$ for each $j=-m,-m+1,\ldots,-1$ and let $N_0\in\N\cup\{\infty\}$. If a sequence $\{d_{i}\}_{i\in I}$ in $[A_{-m},A_{0}]$ satisfies lower exterior majorization by $\{(A_{j},N_{j})\}_{j=-m}^{0}$ and
\begin{equation}\label{nfrhorn1}\sum_{i\in I}d_{i} = \sum_{j=-m}^{-1}N_{j}A_{j},\end{equation}
then there exists a self-adjoint operator $E$ with diagonal $\{d_{i}\}_{i\in I}$ and the following properties:

\begin{align}
\label{nfrhorn2}\sigma(E) &\subseteq \{A_{-m},\ldots,A_{0}\},
\\
\label{nfrhorn3} m_{E}(A_{j}) &= N_{j}\ \text{for each}\ j=-m,\ldots,-1,
\\
\label{nfrhorn4}m_{E}(A_{0}) &= |I| - \sum_{j=-m}^{-1}N_{j}\geq 0.
\end{align}
\end{thm}

To extend this to a general Horn's theorem for finite rank (not necessarily positive) operators we shall apply Theorems \ref{pfrhorn} and \ref{nfrhorn} to the positive and negative terms, respectively. Since the required trace conditions need not be satisfied we have the following ``decoupling lemma''.

\begin{lem}\label{dcup} Let $\{d_{i}\}_{i\in I}$ be a bounded sequence in $\R$ and let $\delta,\gamma>0$. Let $J \subset I$ be a subset such that $d_i \in [-\gamma,\delta]$ for all $i \in J$ and
\begin{equation}\label{dcup0}
\sum_{i\in J,\ d_{i}\leq 0}(d_{i}+\gamma)=\infty \quad\text{or}\quad \sum_{i\in J,\ d_{i}\geq 0}(\delta-d_{i})=\infty.
\end{equation}
Then for any $\eta\geq 0$ the following two hold.

(i) There is a sequence $\{\tilde{d}_{i}\}_{i\in I}$ such that $\tilde d_i \in [-\gamma,\delta]$ for all $i \in J$, $\tilde d_i=d_i$ for $i \in I \setminus J$, and
\begin{equation}\label{dcup1}
\sum_{i \in J, \ \tilde{d}_{i}<0} \tilde{d}_{i} = \sum_{i\in J,\  d_{i}<0}d_{i}- \eta \quad\text{and}\quad \sum_{i\in J,\ \tilde{d}_{i}>0} \tilde{d}_{i} = \sum_{i\in J,\ d_{i}>0}d_{i} + \eta.\end{equation}

(ii) If $\tilde{E}$ is a self-adjoint operator with diagonal $\{\tilde{d}_{i}\}_{i\in I}$, then there exists an operator $E$ unitarily equivalent to $\tilde{E}$ with diagonal $\{d_{i}\}_{i\in I}$. 
\end{lem}

\begin{remark} If the series $\sum_{i\in J,\ d_{i}>0} d_i$ is divergent, then we interpret \eqref{dcup1} to mean that $\sum_{i\in J,\ \tilde{d}_{i}>0}\tilde{d}_{i}$ is also divergent, and similarly for $\sum_{i\in J,\ d_{i}< 0} d_i$.
\end{remark}

\begin{proof} If $\eta=0$ then we simply take $\tilde{d}_{i} = d_{i}$ for all $i\in I$. Thus, we may assume $\eta>0$.

Without loss of generality, we may assume 
\[\sum_{i\in J,\ d_{i}\geq 0}(\delta-d_{i})=\infty,\]
since the other case can be handled by applying the following argument to $\{-d_{i}\}$.

Set $M:=\lceil\eta/\gamma\rceil$ so that
\begin{equation}\label{dcup2}(M-1)\gamma<\eta\leq M\gamma.\end{equation}
Next we will show that there exists  two finite sets $I_{0},I_{1}\subseteq J^{+}:=\{i\in J :d_{i}\geq 0\}$ with the following three properties:
\begin{align}
\label{dcup3}\max_{i\in I_{0}}d_{i} & \leq \min_{i\in I_{1}}d_{i},
\\
\label{dcup4}\sum_{i\in I_{1}}(\delta-d_{i}) & >M\left(\min_{i\in I_{1}}d_{i}\right) + \eta,
\\
\label{dcup5}|I_{0}| &=M.
\end{align}

Since $J^{+}$ is infinite, there is some $x\in[0,\delta]$ such that for all $\eps>0$ the set $\{i\in J^{+}:d_{i}\in(x-\eps,x+\eps)\}$ is infinite.
First, we consider the case that $x=\delta$ is the only such point. By \eqref{dcup0} the set $\{i\in J^{+}:d_{i}<\delta\}$ must be infinite. Let $I_{0}$ be the set of indices of the $M$ smallest terms of $\{d_{i}\}_{i\in J^{+}}$. Since
\[\sum_{i\in J^{+}\setminus I_{0}}(\delta-d_{i})=\infty\]
we can find a finite subset $I_{1}\subset I\setminus I_{0}$ satisfying \eqref{dcup4}.
Next, assume $x\in[0,\delta)$. Then, there exists a sequence $\{i_{n}\}_{n\in\N}$ of distinct elements in $J^+$ such that $\{d_{i_{n}}\}_{i\in\N}$ is monotone and converges to $x$. This gives us two possibilities.

{\bf Case 1:} The sequence $\{d_{i_{n}}\}_{i\in\N}$ is nonincreasing. Since $x<\delta$ we can choose $N\in\N$ such that
\[\sum_{n=1}^{N}(\delta-d_{i_{n}})>Md_{i_{1}}+\eta.\]
Let $I_{1} = \{i_{1},\ldots,i_{N}\}$ and let $I_{0}= \{i_{N+1},\ldots,i_{N+M}\}$. Since $d_{i_{1}}\geq d_{i_{N}}=\min_{i\in I_{1}}d_{i}$, the sets $I_{0}$ and $I_{1}$ clearly satisfy \eqref{dcup3}, \eqref{dcup4} and \eqref{dcup5}.

{\bf Case 2:} The sequence $\{d_{i_{n}}\}_{i\in\N}$ is nondecreasing.
Set $I_{0} = \{i_{1},\ldots,i_{M}\}$. Since $d_{i_{n}}\leq x<\delta$ for all $n\in \N$, we have
\[\sum_{n=M+1}^{\infty}(\delta-d_{i_{n}})=\infty.\]
Thus, we can find a finite set $I_{1}\subset \{i_{M+1},i_{M+2},\ldots\}$ such that
\[\sum_{i\in I_{1}}(\delta-d_{i})>Mx + \eta\geq M\left(\min_{i\in I_{1}}d_{i}\right)+\eta.
\]
Consequently, we have shown the existence of sets $I_{0}$ and $I_{1}$ which satisfy \eqref{dcup3}--\eqref{dcup5}. Set
\[\eta_{0} = \eta + \sum_{i\in I_{0}}d_{i},\]
and note that
\[\sum_{i\in I_{0}}(d_{i}+\gamma) = \sum_{i\in I_{0}}d_{i} + M\gamma \geq \sum_{i\in I_{0}}d_{i} + \eta=\eta_{0}.\]
By \eqref{dcup3}--\eqref{dcup5} we have
\[
\sum_{i\in I_{1}}(\delta-d_{i}) > M\left(\min_{i\in I_{1}}d_{i}\right) + \eta  \geq M \left(\max_{i\in I_{0}}d_{i}\right) + \eta \geq \sum_{i\in I_{0}}d_{i} + \eta = \eta_{0}.\]
Thus, we can apply Lemma \ref{movelemma} (i) to the sequence $\{d_{i}\}_{i\in I}$ with $A=-\gamma$, $B=\delta$, and with $\eta_{0}$ defined as above to obtain a sequence $\{\tilde{d}_{i}\}_{i\in I}$. Observe that Lemma \ref{movelemma} (ii) immediately yields part (ii). Hence, it remains to verify \eqref{dcup1}.

By \eqref{ops2} and \eqref{dcup2} we have
\[\sum_{i\in I_{0}}(\tilde{d}_{i}+\gamma) = \sum_{i\in I_{0}}(d_{i}+\gamma)-\eta_{0} = M\gamma-\eta < \gamma.\]
This implies that $\tilde{d}_{i}<0$ for all $i\in I_{0}$. Thus, by \eqref{ops1} and \eqref{ops2} we have
\[
\sum_{i\in J,\ \tilde{d}_{i}<0}\tilde{d}_{i} =\sum_{i\in J,\ d_{i}<0}d_{i} + \sum_{i\in I_{0}}\tilde{d}_{i} =\sum_{i\in J,\ d_{i}< 0}d_{i} -\eta.\]
Likewise, we have
\begin{equation*}\begin{split}
\sum_{i\in J,\ \tilde{d}_{i}>0}\tilde{d}_{i} &  = \sum_{i\in J\setminus(I_{0}\cup I_{1})}d_{i} + \sum_{i\in I_{1}}\tilde{d}_{i} = \sum_{i\in J\setminus(I_{0}\cup I_{1})}d_{i} + \sum_{i\in I_{1}}d_{i} + \eta_{0}
\\
 & = \sum_{i\in J\setminus  I_{0} }d_{i} + \sum_{i\in I_{0}}d_{i} + \eta = \sum_{i\in J,\ d_{i}>0}d_{i}+ \eta.\\
\end{split}\end{equation*}
The completes the proof of the lemma. 
\end{proof}

Finally, we are ready to prove the Schur-Horn Theorem for general finite rank operators on an infinite dimensional (separable) Hilbert space.

\begin{thm}\label{frhorn} Let $m,p\in\N$ and $\{A_{j}\}_{j=-m}^{p}$ be an increasing sequence with $A_{0}=0$. Let $N_{j}\in\N$ for $-m\le j \le p$, $j \ne 0$, and let $N_{0}=\infty$. Let $\{d_{i}\}_{i=1}^{\infty}$ be a sequence in $[A_{-m},A_{p}]$. There is a finite rank, self-adjoint operator $E$ with diagonal $\{d_{i}\}$ and eigenvalue-multiplicity list $\{(A_{j},N_{j})\}_{j=-m}^{p}$ if and only if $\{d_{i}\}$ satisfies the following three conditions:
\begin{enumerate}
\item lower exterior majorization by $\{(A_{j},N_{j})\}_{j=-m}^{0}$, 
\item upper exterior majorization by $\{(A_{j},N_{j})\}_{j=0}^{p}$, 
\item the trace condition
\begin{equation}\label{frhorn1}\sum_{i=1}^\infty d_{i} = \sum_{j=-m}^{-1}N_{j}A_{j} + \sum_{j=1}^{p}N_{j}A_{j}.\end{equation}
\end{enumerate}
\end{thm}

\begin{proof} First, assume $\{d_{i}\}$ satisfies conditions (i),(ii) and (iii). Set
\[\eta = \sum_{j=1}^{p}N_{j}A_{j} - \sum_{d_{i}\geq 0}d_{i}.\]

Fix any $0<\delta<\min\{A_{1},-A_{-1}\}$ and define $J = \{i\in\N:d_{i}\in[-\gamma,\delta]\}$, where $\gamma=\delta$. Since the sequence $\{d_{i}\}_{i=1}^{\infty}$ is absolutely summable, the sequence $\{d_{i}\}_{i\in J}$ is also an infinite absolutely summable sequence. Thus, \eqref{dcup0} holds and we can apply Lemma \ref{dcup} to obtain a sequence $\{\tilde d_{i}\}_{i=1}^{\infty}$.

By Lemma \ref{dcup} (i) the values of $d_i$ and $\tilde d_i$, which lie outside of the interval $[-\delta,\delta]$, must coincide. Hence, $\{\tilde d_i\}_{i=1}^\infty$ automatically satisfies upper upper exterior majorization \eqref{ulmaj1} for each $r=1,\ldots,p$. To verify the same for $r=0$ we use
\eqref{dcup1}
\begin{align*}
\sum_{\tilde{d}_{i}>0}\tilde{d}_{i} & = \sum_{d_{i}>\delta}d_{i} + \sum_{i\in J,\ \tilde{d}_{i}>0}\tilde{d}_{i} = \sum_{d_{i}>\delta}d_{i} + \sum_{i\in J,\ d_{i}>0}d_{i} + \eta = \sum_{d_{i}\geq 0}d_{i} + \eta = \sum_{j=1}^{p}N_{j}A_{j}.
\end{align*}
Likewise, 
$\{\tilde d_i\}_{i=1}^\infty$ automatically satisfies lower exterior majorization \eqref{llmaj1} for each $r=-m,\ldots,-1$. The same holds for $r=0$ by \eqref{dcup1} and \eqref{frhorn1}
\begin{equation}\label{frhorn4}
\begin{split}
\sum_{\tilde{d}_{i}<0}\tilde{d}_{i} & = \sum_{d_{i}<-\delta}d_{i} + \sum_{i\in J,\ \tilde{d}_{i}<0}\tilde{d}_{i} = \sum_{d_{i}<0}d_{i} - \eta = \sum_{d_{i}<0}d_{i} - \sum_{j=1}^{p}N_{j}A_{j} + \sum_{d_{i}\geq 0}d_{i}\\
 & = \sum_{i=1}^{\infty}d_{i} - \sum_{j=1}^{p}N_{j}A_{j} = \sum_{i=-m}^{-1}N_{j}A_{j}.
\end{split}\end{equation}

Let $I^+=\{i\in \N: \tilde d_i \ge 0\}$ and $I^-=\{i\in \N: \tilde d_i < 0\}$.
The above shows that $\{\tilde{d}_{i}\}_{i\in I^+}$ satisfies upper exterior majorization by $\{(A_{j},N_{j})\}_{j=0}^{p}$ and \eqref{pfrhorn1} holds.
Likewise, $\{\tilde{d}_{i}\}_{i\in I^-}$ satisfies lower exterior majorization by $\{(A_{j}, N_{j})\}_{j=-m}^{0}$ and \eqref{nfrhorn1} holds.
Theorem \ref{pfrhorn} implies that there is an operator $\tilde{E}_{+}$ with $\sigma(\tilde{E}_{+}) \subseteq \{A_{0},A_{1},\ldots,A_{p}\}$, $m_{\tilde{E}_{+}}(A_{j}) = N_{j}$ for all $j=1,\ldots,p$, and diagonal $\{\tilde{d}_{i}\}_{i\in I^+}$. Likewise,
Theorem \ref{nfrhorn} implies that there is an operator $\tilde{E}_{-}$ with $\sigma(\tilde{E}_{-}) \subseteq \{A_{-m},A_{-m+1},\ldots,A_{0}\}$, $m_{\tilde{E}_{-}}(A_{j}) = N_{j}$ for all $j=-m,\ldots,-1$, and diagonal $\{\tilde{d}_{i}\}_{i\in I^-}$.

Since either $I^+$ or $I^-$, or both, are infinite, by \eqref{pfrhorn4} or \eqref{nfrhorn4}, the operator $\tilde{E}=\tilde{E}_{+}\oplus\tilde{E}_{-}$ has eigenvalue-multiplicity list $\{(A_{j},N_{j})\}_{j=-m}^{p}$ and diagonal $\{\tilde{d}_{i}\}_{i\in\N}$. By Lemma \ref{dcup} (ii) there is an operator $E$, unitarily equivalent to $\tilde{E}$ with diagonal $\{d_{i}\}_{i\in \N}$. This completes the proof that (i), (ii) and (iii) are sufficient. 

Conversely, assume that $E$ has eigenvalue-multiplicity list $\{(A_{j},N_{j})\}_{j=-m}^{p}$ and diagonal $\{d_{i}\}$. Parts (i) and (ii) are immediate consequences of Theorems \ref{ext} and \ref{exts}, respectively. On the other hand,  \eqref{frhorn1} follows by considering the trace of $E$.
\end{proof}

\section{Horn's Theorem for operators with finite spectrum}

Horn's Theorem for operators with finite spectrum breaks into three cases depending on the number of infinite multiplicities. We have already considered the case of exactly one infinite multiplicity in Theorem \ref{frhorn}. Theorem \ref{2infhorn} deals with the ``summable'' case of operators with two infinite multiplicities, that is when $C(B/2)$ and $D(B/2)$ are finite. Theorem \ref{3infhorn} shows the sufficiency of the ``non-summable'' case, which includes operators with three or more infinite multiplicities. The combination of these two results yields the sufficiency part of Theorem \ref{fullthm}.

\begin{thm}\label{2infhorn} Let $\{(A_{j},N_{j})\}_{j=-m}^{n+p+1}$ be as in Definition \ref{conv} and, in addition, $N_j<\infty$ for all $j=1,\ldots,n$. Let $\{d_{i}\}_{i\in I}$ be a sequence in $[A_{-m},A_{n+p+1}]$ which satisfies the following four conditions:
\begin{enumerate}
\item lower exterior majorization by $\{(A_{j},N_{j})\}_{j=-m}^{0}$, 
\item upper exterior majorization by $\{(A_{j},N_{j})\}_{j=p+1}^{n+p+1}$, 
\item interior majorization by $\{(A_{j},N_{j})\}_{j=-m}^{n+p+1}$ if $n\ge 1$, and otherwise, if $n=0$, the trace condition (both of them implicitly require that $C(B/2),D(B/2)<\infty$)
\[
\exists k\in \Z \qquad C(B/2) -D(B/2) =  \sum_{\genfrac{}{}{0pt}{}{j=-m}{j \not= 0,1}}^{p+1} A_j N_j +k B,
\]
\item $|\{i: d_{i}<B/2\}| = |\{i:d_i \ge B/2 \}|=\infty$.
\end{enumerate}
Then, there exists a self-adjoint operator $E$ with eigenvalue-multiplicity list $\{(A_{j},N_{j})\}_{j=-m}^{n+p+1}$ and diagonal $\{d_{i}\}_{i\in I}$.
\end{thm}

\begin{proof} We will only deal with the case that $m,p\geq 1$ since the case that $m=p=0$ is Theorem \ref{inthorn} and the case where one of $m$ or $p$ is equal to zero is a slight modification of the argument below.
Let $\delta>0$ such that $2\delta<\min\{-A_{-1},A_{1},B-A_{n},A_{n+2}-B,B\}$. Let
\[J_{0} = \{i:d_{i}\in(-\delta,\delta)\} \quad \text{and} \quad J_{B} = \{i:d_{i}\in (B-\delta,B+\delta)\}.\]

Since $C(B/2),D(B/2)<\infty$, the assumption (iv) implies that $\{d_{i}\}_{i\in J_{0}}$ and $\{B-d_{i}\}_{i\in J_{B}}$ are infinite absolutely summable sequences. Define 
\[\eta_{0} =  \sum_{d_{i}<0}d_{i} - \sum_{j=-m}^{-1}N_{j}A_{j}\quad\text{and}\quad \eta_{B} = \sum_{j=n+2}^{n+p+1}N_{j}(A_{j}-B) - \sum_{d_{i}> B}(d_{i}-B).\]
From \eqref{ulmaj1} and \eqref{llmaj1} we see that $\eta_{0},\eta_{B}\geq 0$.
Applying Lemma \ref{dcup} to $\{d_i\}_{i\in I}$ on the interval $[-\delta,\delta]$ and $J_0 \subset I$, there is a sequence $\{d'_{i}\}_{i\in I}$ such that
\[
\sum_{i\in J_{0},\ d'_{i}<0} d'_{i} = \sum_{i\in J_{0},\ d_{i}<0 }d_{i} - \eta_{0}\quad
\text{and}
\quad\sum_{i\in J_{0},\ d'_{i}>0} d'_{i} = \sum_{i\in J_{0},\ d_{i}>0}d_{i} + \eta_{0}.
\]
Applying again Lemma \ref{dcup} to the sequence $\{d'_{i}-B\}_{i\in I}$ on the interval $[-\delta,\delta]$ and $J_B \subset I$, there is a sequence $\{\tilde{d}_{i}-B\}_{i\in I}$ such that
\[
\sum_{i\in J_{B},\  \tilde{d}_{i}<B}(\tilde{d}_{i}-B) = \sum_{i\in J_{B},\ d_{i}<B}(d_{i}-B) - \eta_{B}
\quad
\text{and}
\quad\sum_{i\in J_{B},\ \tilde{d}_{i}>B }(\tilde{d}_{i}-B) = \sum_{i\in J_{B},\ d_{i}>B} (d_{i}-B) + \eta_{B}.
\]

Observe also that the values of $d_i$ and $d'_i$, which lie outside of the interval $[-\delta,\delta]$, must coincide. The same is true for the values of $d'_i$ and $\tilde d_i$, which lie outside of the interval  $[B-\delta,B+\delta]$.
Thus, we have
\begin{equation}\label{2infhorn1}\sum_{\tilde{d}_{i}<0}\tilde{d}_{i} =
\sum_{d'_{i}<0} d'_{i} 
= \sum_{\genfrac{}{}{0pt}{}{d'_{i}<0}{i\in J_{0}}} d'_{i} + \sum_{\genfrac{}{}{0pt}{}{d_{i}<0}{i\notin J_{0}}} d_{i} = \sum_{d_{i}<0}d_{i} - \eta_{0} = \sum_{j=-m}^{-1}N_{j}A_{j}\end{equation}
and
\begin{equation}\label{2infhorn2}
\sum_{\tilde{d}_{i}>B}(\tilde d_{i}-B) 
= \sum_{\genfrac{}{}{0pt}{}{\tilde{d}_{i}>B}{i\in J_{B}}}(\tilde{d}_{i}-B) +\sum_{\genfrac{}{}{0pt}{}{d_{i}>B}{i\notin J_{B}}}(d_{i}-B) = \sum_{d_{i}>B}(d_{i}-B) + \eta_{B}= \sum_{j=n+2}^{n+p+1}N_{j}(A_{j}-B).\end{equation}

Let $I^-=\{i\in I: \tilde d_i < 0\}$, $I^0=\{i\in I: \tilde d_i \in [0,B]\}$, and $I^+= \{i\in I: \tilde d_i >B \}$.
The above observation and \eqref{2infhorn1} imply that $\{\tilde{d}_{i}\}_{i\in I^-}$ satisfies lower exterior majorization by $\{(A_{j},N_{j})\}_{j=-m}^{0}$ and the trace condition \eqref{nfrhorn1}. Theorem \ref{nfrhorn} implies that there is a self-adjoint operator $\tilde{E}_{-}$ with $\sigma(\tilde{E}_{-}) \subseteq \{A_{-m},\ldots,A_{0}\}$, $m_{\tilde{E}_{-}}(A_{j}) = N_{j}$ for each $j=-m,\ldots,-1$, diagonal $\{\tilde{d}_{i}\}_{i\in I^-}$ and
\begin{equation}
\label{2infhorn3}m_{\tilde{E}_{-}}(A_{0}) = |I^-| - \sum_{j=-m}^{-1}N_{j}.\end{equation}
Likewise, $\{\tilde{d}_{i}\}_{i\in I^+}$ satisfies upper exterior majorization by $\{(A_{j},N_{j})\}_{j=n+1}^{n+p+1}$ as well as the trace condition 
\[
\sum_{i\in I^+}(\tilde d_{i}-B) = \sum_{j=n+1}^{n+p+1}N_{j}(A_{j}-B).
\] 
By Theorem \ref{pfrhorn} there is a self-adjoint operator $\tilde{E}_{+}$ with $\sigma(\tilde{E}_{+}) \subseteq \{A_{n+1},\ldots,A_{n+p+1}\}$, $m_{\tilde{E}_{+}}(A_{j}) = N_{j}$ for $j=n+2,\ldots,n+p+1$, diagonal $\{\tilde{d}_{i}\}_{i\in I^+}$ and
\[m_{\tilde{E}_{+}}(A_{n+1}) = |I^+| - \sum_{j=n+2}^{n+p+2}N_{j}.\]

We claim that the sequence $\{\tilde{d}_{i}\}_{i\in I^0}$ satisfies interior majorization by $\{(A_{j},N_{j})\}_{j=0}^{n+1}$. Indeed, for any $\alpha\in[A_{1},A_{n}]$ we have
\begin{equation}\label{2infhorn7}
\sum_{\tilde{d}_{i}<\alpha}\tilde{d}_{i} = \sum_{\genfrac{}{}{0pt}{}{d_{i}<\alpha}{i\notin J_{0}}}d_{i} + \sum_{\genfrac{}{}{0pt}{}{\tilde{d}_{i}<0}{i\in J_{0}}}\tilde{d}_{i} + \sum_{\genfrac{}{}{0pt}{}{\tilde{d}_{i}>0}{i\in J_{0}}}\tilde{d}_{i} = \sum_{\genfrac{}{}{0pt}{}{d_{i}<\alpha}{i\notin J_{0}}}d_{i} + \sum_{\genfrac{}{}{0pt}{}{d_{i}<0}{i\in J_{0}}}d_{i} - \eta_{0} + \sum_{\genfrac{}{}{0pt}{}{d_{i}>0}{i\in J_{0}}}d_{i} + \eta_{0} =  \sum_{d_{i}<\alpha}d_{i}.\end{equation}
By a similar calculation, for any $\alpha\in[A_{1},A_{n}]$ we have
\begin{equation}\label{2infhorn8}\sum_{\tilde{d}_{i}\geq \alpha}(B-\tilde{d}_{i}) = \sum_{d_{i}\geq \alpha}(B-d_{i}).\end{equation}
%There is some $k_{0}\in\Z$ such that
%\begin{equation}\label{2infhorn9}\sum_{d_{i}<A_{n}}d_{i} - \sum_{d_{i}\geq A_{n}}(B-d_{i}) = \sum_{\genfrac{}{}{0pt}{}{j=-m}{j\neq 0,n+1}}^{n+p+1}A_{j}N_{j} + k_{0}B.\end{equation}
Using \eqref{2infhorn1}, \eqref{2infhorn2}, \eqref{2infhorn7}, and \eqref{2infhorn8} we calculate
\begin{equation}\label{2infhorn5}\begin{split}
\sum_{0\leq \tilde{d}_{i}<A_{n}}\tilde{d}_{i} - \sum_{A_{n}\leq \tilde{d}_{i}\leq B}(B & -\tilde{d}_{i})  = \sum_{\tilde{d}_{i}<A_{n}}\tilde{d}_{i} - \sum_{\tilde{d}_{i}\geq A_{n}}(B-\tilde{d}_{i}) - \sum_{\tilde{d}_{i}<0}\tilde{d}_{i} - \sum_{\tilde{d}_{i}>B}(\tilde{d}_{i}-B)\\
 & = \sum_{d_{i}<A_{n}}d_{i} - \sum_{d_{i}\geq A_{n}}(B-d_{i}) - \sum_{n=-m}^{-1}N_{j}A_{j} - \sum_{j=n+2}^{n+p+1}N_{j}(A_{j}-B)\\
 & = \sum_{j=1}^{n}A_{j}N_{j} + B\left(k_{0}+\sum_{j=n+2}^{n+p+1}N_{j}\right).\\
\end{split}\end{equation}
The last step is a consequence of the trace condition \eqref{fulltrace1} for $\{d_i\}_{i\in I}$. This shows that $\{\tilde{d}_{i}\}_{i\in I^0}$ also satisfies \eqref{fulltrace1} with $k_{0} + \sum_{j=n+2}^{n+p+1}N_{j}$ in the place of $k_{0}$.

Using \eqref{fullmaj1} for the sequence $\{d_{i}\}_{i\in I}$ yields
\begin{equation}\label{2infhorn6}\begin{split}
 \sum_{\tilde{d}_{i}\in[0,A_{r})} & \tilde{d}_{i}  = \sum_{\tilde{d}_{i}<A_{r}}\tilde{d}_{i} - \sum_{\tilde{d}_{i}<0}\tilde{d}_{i} = \sum_{d_{i}<A_{r}}d_{i} - \sum_{j=-m}^{-1}N_{j}A_{j}\\
 & \geq \sum_{\genfrac{}{}{0pt}{}{j=-m}{j\neq 0}}^{r}A_{j}N_{j} + A_{r}\left(k_{0} - |\{i\in I\colon A_{r}\leq d_{i}<A_{n}\}| + \sum_{\genfrac{}{}{0pt}{}{j=r+1}{j\neq n+1}}^{n+p+1}N_{j}\right) - \sum_{j=-m}^{-1}N_{j}A_{j}\\
 & = \sum_{j=1}^{r}A_{j}N_{j} + A_{r}\left(k_{0}+\sum_{j=n+2}^{n+p+1}N_{j} - |\{i\in I\colon A_{r}\leq d_{i}<A_{n}\}| + \sum_{j=r+1}^{n}N_{j}\right).
\end{split}\end{equation}
Together \eqref{2infhorn5} and \eqref{2infhorn6} show that $\{\tilde{d}_{i}\}_{i\in I^0}$ satisfies interior majorization by $\{(A_{j},N_{j})\}_{j=0}^{n+1}$.
By Theorem \ref{inthorn} there is a self-adjoint operator $\tilde{E}_{0}$ with $\sigma(\tilde{E}_{0}) \subseteq \{A_{0},\ldots,A_{n+1}\}$, $m_{\tilde{E}_{0}}(A_{j}) = N_{j}$ for each $j=1,\ldots,n$, and diagonal $\{\tilde{d}_{i}\}_{i\in I^0}$.

Define the operator $\tilde{E}=\tilde{E}_{-}\oplus\tilde{E}_{0}\oplus\tilde{E}_{+}$. One of the sets $\{i:A_{-1}<\tilde{d}_{i}\leq 0\}$ or $\{i:0\leq \tilde{d}_{i}<A_{1}\}$ must be infinite since $\{i:A_{-1}\leq d_{i}<A_{1}\}$ is infinite. If $\{i:A_{-1}\leq \tilde{d}_{i}<0\}$ is infinite, then by \eqref{2infhorn3} $m_{\tilde{E}_{-}}(0)=\infty$. If $\{i:0\leq \tilde{d}_{i}<A_{1}\}$ is infinite, then by Theorem \ref{inthorn} we have $m_{\tilde{E}_{0}}(0)=\infty$. In either case $m_{\tilde{E}}(0)=\infty$. By similar considerations we also have $m_{\tilde{E}}(B)=\infty$.

The operator $\tilde{E}$ has diagonal $\{\tilde{d}_i\}_{i\in I}$ and eigenvalue-multiplicity list $\{(A_{j},N_{j})\}_{j=-m}^{n+p+1}$. Lemma \ref{dcup} (ii), which is technically applied to the operator $\tilde E - B \mathbf I$, implies that there is an operator $E'$, unitarily equivalent to $\tilde{E}$ with diagonal $\{d'_{i}\}_{i\in I}$. Another application of Lemma \ref{dcup} (ii) to $E'$ implies that there is an operator $E$, unitarily equivalent to $E'$, and thus to $\tilde E$, with diagonal $\{d_{i}\}_{i\in I}$. This completes the proof of Theorem \ref{2infhorn}.
\end{proof}

Finally, we need to address the case of operators with three or more eigenvalues with infinite multiplicity. Quite surprisingly, this case is much easier than that of two infinite multiplicities. We shall make use of the following result established by the second author, see \cite[Theorem 4.2]{jj}.

\begin{thm}\label{3intsuff} Suppose that $\Lambda\subset[0,B]$ is a countable set with $0,B\in\Lambda$. Suppose that $N_{\lambda}\in\N\cup\{\infty\}$ for each $\lambda\in \Lambda\cap(0,B)$, and set $N_{0}=N_{B}=\infty$. If $\{d_{i}\}_{i\in I}$ is a sequence in $[0,B]$ such that for some (and hence for all) $\alpha\in (0,B)$
\begin{equation}\label{3intsuff1}
C(\alpha)=\infty\qquad\text{or}\qquad D(\alpha)=\infty,
\end{equation}
then there is a positive diagonalizable operator $E$ with diagonal $\{d_{i}\}_{i\in I}$ and eigenvalues $\Lambda$ with prescribed multiplicities
\[
m_{E}(\lambda)=
\begin{cases}
N_{\lambda} & \lambda\in \Lambda,
\\
0 & \lambda \not\in \Lambda.
\end{cases}
\]
\end{thm}

We are now ready to prove the analogue of Theorem \ref{fullthm} in the non-summable scenario \eqref{3intsuff1}. In particular, Theorem \ref{3infhorn} deals with the case of more than two eigenvalues with infinite multiplicities.

\begin{thm}\label{3infhorn} 
Let $\{(A_{j},N_{j})\}_{j=-m}^{n+p+1}$ be as in Definition \ref{conv}. Let $\{d_{i}\}_{i\in I}$ be a sequence in $[A_{-m},A_{n+p+1}]$ which satisfies the following three conditions:
\begin{enumerate}
\item lower exterior majorization by $\{(A_{j},N_{j})\}_{j=-m}^{0}$, 
\item upper exterior majorization by $\{(A_{j},N_{j})\}_{j=p+1}^{n+p+1}$,
\item for some (and hence for all) $\alpha\in(0,B)$,
\eqref{3intsuff1} holds.
\end{enumerate}
Then, there exists a self-adjoint operator $E$ with eigenvalue-multiplicity list $\{(A_{j},N_{j})\}_{j=-m}^{n+p+1}$ and diagonal $\{d_{i}\}_{i\in I}$.
\end{thm}

\begin{proof} It suffices to consider only the case when $m,p\geq 1$ since the case that $m=p=0$ is Theorem \ref{3intsuff} and the case when either $m=0$ or $p=0$ is an easy modification of the argument below.

Let $\delta>0$ such that $2\delta<\min\{-A_{-1},A_{n+2}-B,B\}$. 
Define the sets
\[I_{0} = \{i:0\leq d_{i}<B/2\} \quad\text{and}\quad I_{B} = \{i:B/2\leq d_{i}\leq B\}.\]
Upper and lower exterior majorization imply that
\[\sum_{d_{i}<0}-d_{i},\sum_{d_{i}>B}(d_{i}-B)<\infty.\]
Thus, (iii) implies that
\[\sum_{i\in I_{0}}d_{i} + \sum_{i\in I_{B}}(B-d_{i}) = \infty.\]
We can find a partition $I_{0}\cup I_{B}$ into three sets $J_{1},J_{2}$ and $J_{3}$ such that
\[\sum_{\genfrac{}{}{0pt}{}{i\in J_{j}}{d_{i}<B/2}}d_{i} + \sum_{\genfrac{}{}{0pt}{}{i\in J_{j}}{d_{i}\geq B/2}}(B-d_{i}) = \infty\quad\text{for }j=1,2,3.\]
Note that we also have
\[\sum_{i\in J_{j}}d_{i} = \sum_{i\in J_{j}}(B-d_{i}) = \infty\quad \text{for }j=1,2,3.\]

Set
\[\eta_{0}=\sum_{d_{i}<0}d_{i} - \sum_{j=-m}^{-1}N_{j}A_{j}\quad\text{and}\quad \eta_{B} = \sum_{j=n+2}^{n+p+1}N_{j}(A_{j}-B) - \sum_{d_{i}>B}(d_{i}-B).\]
From \eqref{ulmaj1} and \eqref{llmaj1} we see that $\eta_{0},\eta_{B}\geq 0$.
Applying Lemma \ref{dcup} (i) to the sequence $\{d_{i}\}_{i\in I}$ on the interval $[-\delta,B]$ and $J_1 \subset I$, we obtain a sequence $\{d'_{i}\}_{i\in I}$ such that
\[
\sum_{i\in J_{1},\ d'_{i}<0} d'_{i} = -\eta_{0}.
\]
Applying Lemma \ref{dcup} (i) again to the sequence $\{d'_{i}-B\}_{i\in I}$ on the interval $[-B,\delta]$ and $J_2 \subset I$, we obtain a sequence $\{\tilde{d}_{i}-B\}_{i\in I}$ such that
\[
\sum_{i\in J_2, \ \tilde{d}_{i}>B} (\tilde{d}_{i}-B) = \eta_{B}.
\]
Observe also that the values of $d_i$ and $d'_i$, which lie outside of the interval $[-\delta,B]$, must coincide. The same is true for the values of $d'_i$ and $\tilde d_i$, which lie outside of the interval  $[0,B+\delta]$.

Let $I^-=\{i\in I: \tilde d_i < 0\}$, $I^0=\{i\in I: \tilde d_i \in [0,B]\}$, and $I^+= \{i\in I: \tilde d_i >B \}$.
The above observation implies that $\{\tilde{d}_{i}\}_{i\in I^-}$ satisfies lower exterior majorization by $\{(A_{j},N_{j})\}_{j=-m}^{0}$. Moreover, by the choice of $\eta_0$ we have the trace condition 
\[
\sum_{i\in I^-}\tilde{d}_{i} = \sum_{d_{i}<0}d_{i} + \sum_{i\in J_{1},\ \tilde{d}_{i}<0} \tilde{d}_{i} = \sum_{d_i<0}d_{i} + \sum_{i\in J_{1},\ d'_{i}<0} d'_{i} = \sum_{j=-m}^{-1}N_{j}A_{j}.\]
Theorem \ref{nfrhorn} implies that there is a self-adjoint operator $\tilde{E}_{-}$ with $\sigma(\tilde{E}_{-}) \subseteq \{A_{-m},\ldots,A_{0}\}$, $m_{\tilde{E}_{-}}(A_{j}) = N_{j}$ for each $j=-m,\ldots,-1$, and diagonal $\{\tilde{d}_{i}\}_{i\in I^-}$.
Likewise, $\{\tilde{d}_{i}\}_{i\in I^+}$ satisfies upper exterior majorization by $\{(A_{j},N_{j})\}_{j=n+1}^{n+p+1}$ as well as the trace condition 
\begin{align*}
\sum_{i\in I^+}(\tilde{d}_{i}-B)  
= \sum_{d'_{i}>B}(d'_{i}-B) + \sum_{i\in J_{2},\ \tilde{d}_{i}>B}(\tilde{d}_{i}-B) 
 =  \sum_{d_{i}>B}(d_{i}-B)+\eta_B = \sum_{j=n+2}^{n+p+1}N_{j}(A_{j}-B).
\end{align*}
By Theorem \ref{pfrhorn} there is a self-adjoint operator $\tilde{E}_{+}$ with $\sigma(\tilde{E}_{+}) \subseteq \{A_{n+1},\ldots,A_{n+p+1}\}$, $m_{\tilde{E}_{+}}(A_{j}) = N_{j}$ for $j=n+2,\ldots,n+p+1$, and diagonal $\{\tilde{d}_{i}\}_{i\in I^+}$.

Finally, the sequence $\{\tilde{d}_{i}\}_{i\in I^0}$ satisfies \eqref{3intsuff1} since $J_{3} \subset I^0$. By Theorem \ref{3intsuff} there is a self adjoint operator $\tilde{E}_{0}$ with $\sigma(\tilde{E}_{0}) = \{A_{0},A_{1},\ldots,A_{n+1}\}$, $m_{\tilde{E}_{0}}(A_{j})=N_{j}$ for $j=0,\ldots,n+1$ and diagonal $\{\tilde{d}_{i}\}_{i\in I^0}$. Consequently, $\tilde{E} = \tilde{E}_{-}\oplus\tilde{E}_{0}\oplus\tilde{E}_{+}$ has the desired eigenvalues and multiplicities and diagonal $\{\tilde{d}_{i}\}_{i\in I}$. By two applications of Lemma \ref{dcup} (ii) there is an operator $E$, unitarily equivalent to $\tilde{E}$ with diagonal $\{d_{i}\}_{i\in I}$.
\end{proof}

We can now complete the proof of Theorem \ref{fullthm}.

\begin{proof}[Proof of sufficiency of Theorem \ref{fullthm}]
Assume that (i)--(iii) in Theorem \ref{fullthm} hold. If $C(B/2)<\infty$ and $D(B/2)<\infty$, then Theorem \ref{2infhorn} applies. Otherwise, Theorem \ref{3infhorn} applies. In either case, there exists a self-adjoint operator $E$ with eigenvalue-multiplicity list $\{(A_{j},N_{j})\}_{j=-m}^{n+p+1}$ and diagonal $\{d_{i}\}_{i\in I}$.
\end{proof}

\section{Applications of the main result}

The main results of the paper, Theorems \ref{fullthm} and \ref{frhorn}, characterize diagonals of self-adjoint with prescribed finite spectrum and multiplicities. However, these results do not resemble in an obvious way their finite dimensional progenitor, the Schur-Horn Theorem \ref{horn}. In this section we shall consider a converse problem of characterizing spectra of operators with a fixed diagonal. In particular, we shall establish Theorem \ref{La}  which resembles quite closely the Schur-Horn Theorem \ref{horn}. 

For the sake of simplicity we shall concentrate on operators $E$ with finite spectrum such that their smallest and largest eigenvalue have infinite multiplicities. Moreover, by a normalization we can assume that $\{0,1\} \subset \sigma(E) \subset [0,1]$. This is not a true limitation since the part of an operator lying outside of eigenvalues with infinite multiplicities is finite rank, and hence susceptible to the usual majorization techniques. Hence, we avoid dealing with the  exterior majorization condition and instead we concentrate on truly infinite dimensional interior majorization condition. 

Given a self-adjoint operator $E$ with $\sigma(E) \subset [0,1]$, let 
\[
\mathcal H_E = (\ker E \oplus \ker (\mathbf I -E))^\perp.
\]
For a fixed sequence $\{d_{i}\}_{i\in I}$ in $[0,1]$, where $I$ is countable, we consider the set
\begin{multline}\label{LN}
{\Lambda}_{N}(\{d_{i}\})=\big\{(\lambda_{1},\ldots,\lambda_{N})\in(0,1)^{N}:\exists\,E\geq 0\text{  with diagonal $\{d_{i}\}$ such that}
\\ \dim \mathcal H_E=N
\text{ and } E|_{\mathcal H_E} \text{ has eigenvalues } \{\la_i\}_{i=1}^N \text{ listed with multiplicities} \big\}.
\end{multline}
Similarly defined sets, though with ignored multiplicities,
\begin{multline*}
\mathcal{A}_{N}(\{d_{i}\})= \big\{(A_{1},\ldots,A_{N})\in(0,1)^{N}:\forall_{j \ne k} \ A_j \ne A_k,
\ 
\exists\text{ $E\geq 0$ with diagonal $\{d_{i}\}$ and }
\\
 \sigma(E)=\{0,A_{1},\ldots,A_{N},1\}\big\}.
\end{multline*}
were studied in \cite{jj} and \cite{mbjj3}.  
If there exists $\alpha\in(0,1)$ such that
\begin{equation}\label{nF}
\sum_{d_{i}<\alpha}d_{i} + \sum_{d_{i}\geq\alpha}(1-d_{i}) = \infty,
\end{equation}
then by Theorem \ref{3intsuff}, $\Lambda_{N}(\{d_{i}\}) = (0,1)^{N}$. Thus, we will consider only sequences in the set 
\begin{equation}
\label{calf}
\mathcal{F}:=\left\{\{d_{i}\}:\exists\,\alpha\in(0,1)\text{ such that }\sum_{d_{i}<\alpha}d_{i} + \sum_{d_{i}\geq\alpha}(1-d_{i})<\infty\right\}.
\end{equation}
For any sequence $\mathbf d=\{d_i\}_{i\in I}\in\mathcal F$ we define a function $f_{\mathbf d}: (0,1) \to (0,\infty)$ by
\begin{equation}\label{f}
f_{\mathbf d}(\alpha) = (1-\alpha) C(\alpha)+ \alpha D(\alpha),
\qquad\text{where }C(\alpha)=\sum_{d_{i}<\alpha}d_{i}, \ D(\alpha)= \sum_{d_{i}\geq\alpha}(1-d_{i}).
\end{equation}
As a consequence of Theorem \ref{fullthm} we have the following lemma.

\begin{lem}\label{eqv}
Let $\mathbf d=\{d_{i}\}_{i\in I} \in \mathcal F$ be a sequence in $[0,1]$. 
Let $\boldsymbol \la = \{\lambda_{i}\}_{i=1}^N$ be a sequence in $(0,1)$. Then $\boldsymbol \la \in \Lambda_N(\{d_i\})$ $\iff$ 
\begin{align}\label{f1}
C(1/2)-D(1/2) & \equiv \sum_{i=1}^N \lambda_i\mod 1 
\\
\label{f2}
\text{and}\qquad
f_{\mathbf d}(\la_i) & \ge f_{\boldsymbol \lambda} (\lambda_i)\qquad\text{for }i=1,\ldots,N.
\end{align}
\end{lem}

\begin{proof} By Lemma \ref{decomp} the set $\Lambda_N(\mathbf d)$ does not change when we add or remove $0$'s or $1$'s to the sequence $\mathbf d$. Thus, without loss of generality
we can assume that there are infinitely many $d_i<1/2$ and infinitely many $d_i \ge 1/2$. Thus, \eqref{fullinf} holds. Clearly, \eqref{f1} is equivalent with \eqref{fulltrace}, where $(A_j,N_j)_{j=1}^n$ is an eigenvalue-multiplicity list corresponding to eigenvalues $\{\la_i\}_{i=1}^N$ and $m=p=0$. Likewise, \eqref{f2} is equivalent to \eqref{fullmaj}. Applying Theorem \ref{fullthm} yields the lemma.
\end{proof}

The main result of this section, Theorem \ref{La}, shows that $\Lambda_N(\{d_i\})$ is the union of $N-1$ or $N$ upper subsets of constant trace each having a unique minimal element with respect to the majorization order $\prec$, see \cite{moa}. We say that $\{\la_i\}_{i=1}^N=\boldsymbol \la \prec \boldsymbol \mu = \{\mu_i\}_{i=1}^N$ if
\[
\sum_{i=1}^N \lambda_{[i]} = \sum_{i=1}^N \mu_{[i]}
\qquad\text{and}\qquad
\sum_{i=1}^n \lambda_{[i]} \le \sum_{i=1}^n \mu_{[i]}
\qquad\text{for all }1\le n \le N.
\]
In the above $\{\la_{[i]}\}_{i=1}^N$ and $\{\mu_{[i]}\}_{i=1}^N$ denote decreasing rearrangements of $\boldsymbol \la$ and $\boldsymbol \mu$, resp. The relation $\prec$ is a partial order once we identify sequences with the same decreasing rearrangements.

\begin{thm}\label{La}
Let $\{d_{i}\}_{i\in I} \in \mathcal F$ be a sequence in $[0,1]$ such that
\begin{equation}\label{La1}
|\{i \in I: 0<d_i<1/2\}| = |\{i\in I: 1/2 \le d_i < 1\}| = \infty.
\end{equation}
The set $\Lambda_N(\{d_i\})$, where $N\in \N$, has exactly $N$ minimal elements, or $N-1$ minimal elements if $\{d_i\}_{i\in I}$ is a diagonal of a projection, with respect to the majorization order $\prec$. That is, there exist $0 \le \eta <1$ and $\boldsymbol \mu^k=(\mu^k_1, \ldots \mu^k_N) \in (0,1)^N$, $k=0,\ldots,N-1$, such that
\begin{equation}\label{La2}
\Lambda_N(\{d_i\})= \bigcup_{k=0}^{N-1} \{ \boldsymbol \la \in (0,1)^N: \boldsymbol \mu^k \prec \boldsymbol \la \}, \qquad
\sum_{i=1}^N \mu^k_i= k+ \eta.
\end{equation}
In the special case when $\{d_i\}_{i\in I}$ is a diagonal of a projection we have $\eta=0$ and one fewer $\boldsymbol \mu^k$, $k=1,\ldots, N-1$.
\end{thm}

The proof of Theorem \ref{La} relies on two lemmas. Lemma \ref{app} enables us to approximate general sequences in $\mathcal F$ by those with finitely many terms in $(0,1)$. Lemma \ref{fm} shows the existence of unique minimal elements of certain upper sets with respect to $\prec$, which are stable under perturbations.

\begin{lem}\label{app}
Let $\{d_{i}\}_{i\in I} \in \mathcal F$ be a sequence in $[0,1]$ and let $\ve>0$. Then, there exists a sequence $\{\tilde d_{i}\}_{i\in I}$ in $[0,1]$ such that:
\begin{enumerate}
\item $\tilde d_i=0$ or $\tilde d_i=1$ for all but finitely many $i\in I$,
\item $I_0:=\{i\in I: d_i<\ve\}= \{i \in I: \tilde d_i<\ve\}$,
\item $I_1:=\{i\in I: d_i>1-\ve\}=\{ i\in I: \tilde d_i >1-\ve\}$,
\item $d_i=\tilde d_i$ for all $i\in I\setminus (I_0 \cup I_1)$,
\item $\sum_{i\in I_0 }d_i = \sum_{i\in I_0} \tilde d_i$ and $\sum_{i\in I_1 }(1-d_i) = \sum_{i\in I_1} (1-\tilde d_i)$. 
\end{enumerate}
Moreover, for any $N\in\N$ we have
\begin{equation}\label{app2}
\Lambda_N(\{d_i\}) \cap [\ve, 1-\ve]^N = 
\Lambda_N(\{\tilde d_i\}) \cap [\ve, 1-\ve]^N.
\end{equation}
\end{lem}

\begin{proof}
Choose $i_0,i_1\in I$ such that $d_{i_0} \in (0,\ve)$ and $d_{i_1} \in (1-\ve,1)$. Let $\delta=\min(\ve-d_{i_0}, d_{i_1}-(1-\ve))$. Let $0<\ve'<\ve$ be such that 
\[
\sum_{i\in I_0'} d_i + \sum_{i\in I_1'} (1-d_i)< \delta \qquad\text{where }I'_0=\{i\in I: d_i<\ve'\},
I'_1=\{i\in I: d_i>1-\ve'\}.
\]
Define the sequence $\{\tilde d_i\}_{i\in I}$ by
\[
\tilde d_i = \begin{cases} 0 & i\in I'_0,
\\
d_{i_0}+ \sum_{i\in I_0'} d_i & i=i_0,
\\
d_{i_1}-\sum_{i\in I_1'} (1-d_i) & i=i_1,
\\
1 & i\in I'_1,
\\
d_i & \text{otherwise.}
\end{cases}
\]
Observe that $\tilde d_{i_0} \in (0,\ve)$ and $\tilde d_{i_1} \in (1-\ve,1)$. Then, it is straightforward to check that (i)--(v) hold. Consequently, for any $\alpha \in [\ve,1-\ve]$, the quantities $C(\alpha)$ and $D(\alpha)$ stay the same for both sequences $\mathbf d=\{d_i\}_{i\in I}$ and $\tilde {\mathbf d}=\{\tilde d_i\}_{i\in I}$. Thus, functions $f_{\mathbf d}$ and $f_{\tilde {\mathbf d}}$ defined by \eqref{f} agree on $[\ve, 1-\ve]$. Applying Lemma \ref{eqv} yields \eqref{app2}.
\end{proof}

\begin{lem}\label{fm}
Let $\{d_i\}_{i=1}^M $ be a nonincreasing sequence in $[0,1]$, and let $K\in \N_0$ and $0\le \eta<1$ be such that 
\begin{equation}\label{fm1}
\sum_{i=1}^M d_i = K +\eta.
\end{equation}
Suppose $N\in \N$ and $k\in \N_0$ are such that $N< M$, $k \le K$, $k+\eta>0$, and $K-k \le M-N$.
Then, 
\begin{multline}\label{fm2}
\{ \boldsymbol \la \in (0,1)^N: (d_1,\ldots,d_M) \prec (\underbrace{1,\ldots,1}_{K-k},\lambda_1,\ldots,\lambda_N,\underbrace{0,\ldots \ldots,0}_{M-N-(K-k)}) \}
\\
=\{ \boldsymbol \la  \in (0,1)^N: \boldsymbol \mu \prec \boldsymbol \lambda\}\quad\text{for some }\boldsymbol\mu\in(0,1)^{N}\text{ with } \sum_{i=1}^{N}\mu_{i} = k+\eta.
\end{multline}
Moreover, suppose there exist $\ve>0$ and $m_0,m_1\in \N$ such that the following hold:
\begin{equation}\label{fm2a}
d_i \in (\ve, 1-\ve) \iff m_0 \le i \le m_1,
\end{equation}
\begin{equation}\label{fm2b}
|\{i : d_i \in ((k+\eta)/N,1-\ve) \}| \ge N, 
\end{equation}
\begin{equation}\label{fm2c}
|\{i : d_i \in (\ve,(k+\eta)/N) \}| \ge N.
\end{equation}
If $\{d'_i\}_{i=1}^{M'}$, $M' \ge M$, is a nonincreasing sequence in $[0,1]$ such that
\begin{equation}\label{fm3b}
\exists \, n\in \N_0 \quad d'_i \in (\ve, 1-\ve) \iff n+m_0\le i \le n+m_1,
\end{equation}
\begin{equation}\label{fm3a}
d'_i=d_{i-n} \quad\text{for }n+m_0\le i \le n+m_1,
\end{equation}
\begin{equation}\label{fm3c}
\sum_{d_{i}\leq \ve }d_{i} = \sum_{d_{i}'\leq \ve}d_{i}',
\end{equation}
\begin{equation}\label{fm3d}
\sum_{d_{i}\geq 1-\ve}(1-d_{i}) = \sum_{d_{i}'\geq 1-\ve}(1-d_{i}'),
\end{equation}
then \eqref{fm2} holds  with $\{d_i\}_{i=1}^{M}$ replaced by $\{d'_i\}_{i=1}^{M'}$ and the same $\boldsymbol \mu \in (\ve,1-\ve)^N$.
\end{lem}

\begin{proof} Given any $\boldsymbol\lambda = (\lambda_{1},\ldots,\lambda_{N})\in(0,1)^{N}$ we define
\[\overline{\boldsymbol{\lambda}} = (\underbrace{1,\ldots,1}_{K-k},\lambda_1,\ldots,\lambda_N,\underbrace{0,\ldots \ldots,0}_{M-N-(K-k)}).\] 
Define the set
\[\Lambda_{N}^{k} = \{ \boldsymbol \la \in (0,1)^N: (d_1,\ldots,d_M) \prec \overline{\boldsymbol{\lambda}} \}.\]
For $\boldsymbol \mu\in(0,1)^{N}$ define the upper set
\[U(\boldsymbol \mu) = \{\boldsymbol \lambda\in(0,1)^{N}: \boldsymbol\mu \prec \boldsymbol \lambda\}.\]
Our first goal is to find $\boldsymbol\mu$ such that \eqref{fm2} holds, that is, $U(\boldsymbol\mu)=\Lambda_{N}^{k}$. In order to do this we define the functions
\[g(x) = \sum_{\genfrac{}{}{0pt}{1}{i> K-k}{d_{i}>x}}(d_{i}-x)\quad\text{and}\quad h(x)=\sum_{\genfrac{}{}{0pt}{1}{i< K-k+N+1}{d_{i}<x}}(x-d_{i}).\]
Fix $s\in\N$ such that 
\begin{equation}\label{fm3.0}
d_{i}> (k+\eta)/N \iff i\leq s.
\end{equation}

{\bf Case 1.} Assume
\begin{equation}
\label{fm3}g((k+\eta)/N)\leq \sum_{i=1}^{K-k}(1-d_{i}).
\end{equation}
In this case we define $\boldsymbol\mu$ by $\mu_{i} = (k+\eta)/N$, $i=1,\ldots,N$. It is clear that $U(\boldsymbol\mu)$ is the set of all $\boldsymbol\lambda\in(0,1)^{N}$ that sum to $k+\eta$. Thus, by the transitivity of $\prec$, it is enough to show that $\mathbf{d}\prec\overline{\boldsymbol{\mu}}$.
Rearranging \eqref{fm3} gives
\begin{equation}\label{fm3.1}\sum_{i=1}^{s}d_{i}\leq K-k+(s-(K-k))\frac{k+\eta}{N} \le \sum_{i=1}^{s}\overline{\mu}_{i}.\end{equation}
Thus, by \eqref{fm3.0} and \eqref{fm3.1} we have
\[
\sum_{i=1}^{n}d_{i}\leq \sum_{i=1}^{n}\overline{\mu}_{i}
\qquad\text{for } n\leq K-k+N.
\]
Finally, for $n>K-k+N$ we have
\[
\sum_{i=1}^{n}d_{i}\leq \sum_{i=1}^{M}d_{i} = \sum_{i=1}^{M}\overline{\mu}_{i} = \sum_{i=1}^{n}\overline{\mu}_{i}.\]
This completes the proof of the first case.

{\bf Case 2.} Assume 
\begin{equation}\label{fm4}h((k+\eta)/N)\leq \sum_{i=K-k+N+1}^{M}d_{i}.\end{equation}
The proof that the same $\boldsymbol\mu$ works is analogous to Case 1.

{\bf Case 3.} Assume that both \eqref{fm3} and \eqref{fm4} fail. In this case we shall describe the procedure of ``moving'' the terms of $\mathbf d$ toward the global minimum element $((k+\eta)/N, \ldots, (k+\eta)/N)$. In the process, the largest $K-k$ terms of $\mathbf d$ become $1$'s creating  momentum \eqref{fm5} for the next elements $d_i$, $K-k<i<s$, to move down toward the global minimum $(k+\eta)/N$ starting with the largest and coalescing into a cohort until the momentum is exhausted. Likewise, the smallest $M-N-(K-k)$ terms of $\mathbf d$ become $0$'s enabling the elements  $d_i$, $s\le i\le N+(K-k)$, to move up toward the global minimum $(k+\eta)/N$. Unlike Case 1 or 2, we shall not reach this global minimum, but instead a minimal element $\boldsymbol \mu$ defined below.

Since both $g$ and $h$ are continuous and strictly monotone on the interval $[d_{K-k+N},d_{K-k+1}]$, there exist unique numbers $a$ and $b$ with $d_{K-k+N}\leq b<(k+\eta)/N<a\leq d_{K-k+1}$ such that
\begin{equation}\label{fm5}g(a) = \sum_{i=1}^{K-k}(1-d_{i})\end{equation}
and
\begin{equation}\label{fm6}h(b) = \sum_{i=K-k+N+1}^{M}d_{i}.\end{equation}
Define the integers
\begin{equation}\label{fmN}N_{a} = |\{i\colon d_{i}>a\text{ and } i>K-k\}|\quad\text{and}\quad N_{b} = |\{i\colon d_{i}<b\text{ and }i<K-k+N+1\}|.\end{equation}
Finally, define $\boldsymbol\mu$ by
\begin{equation}\label{fmmu}\mu_{i} = 
\begin{cases} a & i=1,\ldots,N_{a},
\\ d_{i+K-k} & i=N_{a}+1,\ldots,N-N_{b},
\\ b & i=N-N_{b}+1,\ldots,N.
\end{cases}\end{equation}

First, we wish to show that $\mathbf d\prec\overline{\boldsymbol\mu}$, that is
\begin{equation}\label{fm7}\sum_{i=1}^{m}(\overline{\mu}_{i}-d_{i})\geq 0\end{equation}
for $m=1,\ldots,M$, with equality at $m=M$.
From \eqref{fm5} and the observation that $\overline{\mu}_{i} = d_{i}$ for $i=K-k+N_{a}+1,\ldots,K-k+N-N_{b}$, we see that 
\begin{equation}\label{fm8}\sum_{i=1}^{m}(\overline{\mu}_{i}-d_{i})=0\quad\text{for }m=K-k+N_{a},\ldots,K-k+N-N_{b}.\end{equation}
From \eqref{fm6} we see that
\[\sum_{i=K-k+N-N_{b}+1}^{M}(\overline{\mu}_{i}-d_{i})=0.\]
Putting these together shows the equality in \eqref{fm7} for $m=M$.

Next, note that $\overline{\mu}_{i}-d_{i}\geq 0$ for $i=1,\ldots K-k$ and $\overline{\mu}_{i}-d_{i}\leq 0$ for $i=K-k+1,\ldots,K-k+N_{a}$. Together with \eqref{fm8} this shows that \eqref{fm7} holds for $m=1,\ldots,K-k+N-N_{b}$. For $i=K-k+N-N_{b}+1,\ldots,K-k+N$ we have $\overline{\mu}_{i}-d_{i}\geq 0$ and for $i\geq K-k+N+1$ we have $\overline{\mu}_{i}-d_{i}=-d_{i}\leq 0$. Since we already know that \eqref{fm7} holds for $m=K-k+N-N_{b}$ and $m=M$, these inequalities show that \eqref{fm7} holds for $m\geq K-k+N-N_{b}+1$. This proves $\mathbf d\prec \overline{\boldsymbol\mu}$, i.e., $\boldsymbol\mu\in\Lambda_{N}^{k}$.
If $\boldsymbol\lambda\in U(\boldsymbol\mu)$, then $\overline{\boldsymbol\mu}\prec\overline{\boldsymbol\lambda}$. The transitivity of $\prec$ implies $\mathbf{d}\prec\overline{\boldsymbol\lambda}$, and thus $\boldsymbol\lambda\in\Lambda_{N}^{k}$. Thus, we have shown that $U(\boldsymbol\mu)\subset\Lambda_{N}^{k}$.

To prove the converse inclusion let $\boldsymbol\lambda\in \Lambda_{N}^{k}$. Without loss of generality we can assume $\lambda_{1}\geq\ldots\geq \lambda_{N}$. It remains to show that $\boldsymbol\mu\prec\boldsymbol\lambda$. 
From \eqref{fm5} the fact that $\mathbf{d}\prec\overline{\boldsymbol\lambda}$ we have
\[\sum_{i=K-k+1}^{K-k+N_{a}}(d_{i}-\overline{\mu}_{i}) = g(a) = \sum_{i=1}^{K-k}(1-d_{i}) = \sum_{i=1}^{K-k}(\overline{\lambda}_{i}-d_{i})\geq -\sum_{i=K-k+1}^{K-k+N_{a}}(\overline{\lambda}_{i}-d_{i}).\]
Rearranging gives
\[
\sum_{i=K-k+1}^{K-k+N_{a}}(\overline{\lambda}_{i}-\overline{\mu}_{i}) =\sum_{i=1}^{N_{a}}(\lambda_{i}-\mu_{i}) \geq 0.\]
Since $\mu_{i}=a$ for $i\leq N_{a}$ we deduce that
\begin{equation}\label{fm9}\sum_{i=1}^{m}\lambda_{i}\geq \sum_{i=1}^{m}\mu_{i}\quad\text{for }m=1,\ldots,N_{a}.\end{equation}

Using \eqref{fm6} and the fact that $\mathbf{d}\prec\overline{\boldsymbol\lambda}$ we have
\[
\sum_{i=K-k+N-N_{b}+1}^{K-k+N}(\overline{\mu}_{i}-d_{i}) = h(b) = \sum_{i=K-k+N+1}^{M}d_{i} = \sum_{i=K-k+N+1}^{M}(d_{i}-\overline{\lambda}_{i})\geq -\sum_{i=K-k+N-N_{b}+1}^{K-k+N}(d_{i}-\overline{\lambda}_{i}).\]
Rearranging this and using the fact that $\boldsymbol \mu$ and $\boldsymbol \lambda$ have the same sum yields
\[
\sum_{i=K-k+N-N_{b}+1}^{K-k+N}(\overline{\mu}_{i}-\overline{\lambda}_{i}) =
\sum_{i=N-N_b+1}^{N}(\mu_{i} - \lambda_{i}) = \sum_{i=1}^{N-N_b} (\lambda_i - \mu_i) \geq 0.
\]
Since $\mu_{i}=b$ for $i=N-N_{b}+1,\ldots,N$, we deduce that
\begin{equation}\label{fm10}\sum_{i=1}^{m}\lambda_{i}\geq \sum_{i=1}^{m}\mu_{i}\quad\text{for }m=N-N_{b},\ldots,N.\end{equation}
By \eqref{fm5} we have
\[\sum_{i=1}^{K-k+N_{a}}d_{i} = \sum_{i=1}^{K-k+N_{a}}\overline{\mu}_{i}.\]
Thus, by the fact that  $\overline{\mu}_{i} = d_{i}$ for $i=K-k+N_{a}+1,\ldots,K-k+N-N_{b}$ and $\mathbf d \prec \overline {\boldsymbol\lambda}$ we deduce that
\begin{equation}\label{fm11} \sum_{i=1}^{m}\lambda_{i}\geq \sum_{i=1}^{m}\mu_{i}\quad\text{for }m=N_{a}+1,\ldots,N-N_{b}.\end{equation}
Putting together \eqref{fm9}, \eqref{fm10}, and \eqref{fm11} shows that $\boldsymbol{\mu}\prec\boldsymbol{\lambda}$, i.e. $\boldsymbol\lambda\in U(\boldsymbol\mu)$. Thus, we have $U(\boldsymbol\mu) = \Lambda_{N}^{k}$, which completes the proof of the first part of the lemma.

Next, we assume there exist $m_{0},m_{1},n$ and $\{d_{i}'\}_{i=1}^{M'}$ satisfying \eqref{fm2a}--\eqref{fm3d}. Using \eqref{fm3b}--\eqref{fm3d} we see that
\begin{align*}
\sum_{i=1}^{M'}d_{i}' &  = \sum_{i=1}^{n+m_{0}-1}(d_{i}'-1) + (n+m_{0}-1) + \sum_{i=n+m_{0}}^{n+m_{1}}d'_{i} + \sum_{i=n+m_{1}+1}^{M'}d_{i}'\\
 &  = \sum_{i=1}^{m_{0}-1}(d_{i}-1) + (n+m_{0}-1) + \sum_{i=m_{0}}^{m_{1}}d_{i} + \sum_{i=m_{1}+1}^{M}d_{i} = \sum_{i=1}^{M}d_{i} + n = K+n+\eta.
 \end{align*}
Define the functions $\tilde{g}$ and $\tilde{h}$ by
\[\tilde{g}(x) = \sum_{\genfrac{}{}{0pt}{1}{i>K+n-k}{d_{i}'>x}}\left(d_{i}'-x\right) \quad\text{and}\quad \tilde{h}(x) = \sum_{\genfrac{}{}{0pt}{1}{i<K+n-k+N}{d_{i}'< x}}(x-d_{i}').\]
Then \eqref{fm3b}--\eqref{fm3d} imply that $\tilde{g}(x)=g(x)$ and $\tilde{h}(x)=h(x)$ for $x\in(\eps,1-\eps)$.

If \eqref{fm3} holds, then by \eqref{fm2a}, \eqref{fm3b}, \eqref{fm3a}, \eqref{fm3d}, and \eqref{fm3.1} we have
\[\sum_{i=1}^{s+n}d_{i}' = n+\sum_{i=1}^{s}d_{i} \leq n+K-k + (s-(K-k))\frac{k+\eta}{N}.\]
Rearranging this inequality gives
\[\tilde{g}((k+\eta)/N)\leq \sum_{i=1}^{K+n-k}d_{i}'.\]
Thus, Case 1 applied to $\{d_{i}'\}_{i=1}^{M'}$ shows that \eqref{fm2} holds with $\{d_{i}'\}$ in place of $\{d_{i}\}$ and the same $\boldsymbol \mu=((k+\eta)/N, \ldots, (k+\eta)/N)$.
A similar argument shows that if \eqref{fm4} holds, then
\[\tilde{h}((k+\eta)/N)\leq \sum_{i=K+n-k+N}^{M'}d_{i}'.\]
Thus, Case 2 shows that \eqref{fm2} holds with $\{d_{i}'\}$ in place of $\{d_{i}\}$ and the same $\boldsymbol \mu$.

Finally, assume both \eqref{fm3} and \eqref{fm4} fail. Since $g((k+\eta)/N)$ and $h((k+\eta)/N)$ are positive, we see that the sequence $\{d_{i}\}_{i=K-k+1}^{K-k+N}$ includes $N$ terms, with at least one term larger than $(k+\eta)/N$ and at least one term smaller than $(k+\eta)/N$. Since $\{d_{i}\}_{i=m_{0}}^{m_{1}}$ includes at least $2N$ terms, with $N$ terms larger than $(k+\eta)/N$ and $N$ terms smaller than $(k+\eta)/N$, we must have $m_{0}\leq K-k+1\leq K-k+N\leq m_{1}$.

Since $K-k+1\geq m_{0}$, we have $1-\eps> d_{K-k+1}\geq a$. This yields
\[\tilde{g}(a) = g(a) = \sum_{i=1}^{K-k}(1-d_{i}) = \sum_{i=1}^{K+n-k}(1-d_{i}').\]
Similarly, since $m_{1}\geq K-k+N$ we have $\eps<d_{K-k+N}\leq b$. This yields
\[\tilde{h}(b) = h(b) = \sum_{i=K-k+N+1}^{M}d_{i} = \sum_{i=K+n-k+N+1}^{M'}d_{i}' .\]
Finally, note that
\[N_{a} = |\{i\colon d_{i}'>a\text{ and } i>K+n-k\}|\quad\text{and}\quad N_{b} = |\{i\colon d_{i}'<b\text{ and }i<K+n-k+N+1\}|,\]
where $N_{a}$ and $N_{b}$ are the numbers in \eqref{fmN}. Thus, applying the definition \eqref{fmmu} to $\{d_{i}'\}$ we obtain the same $\boldsymbol\mu$ as we did for $\{d_{i}\}$. Therefore,  Case 3 shows that \eqref{fm2} holds with $\{d_{i}\}_{i=1}^{M}$ replaced by $\{d_{i}'\}_{i=1}^{M'}$. This completes the proof of Lemma \ref{fm}.
\end{proof}

\begin{proof}[Proof of Theorem \ref{La}] By our assumption \eqref{calf} for any $0<\alpha<1$ we have
\[
C(\alpha) = \sum_{d_{i}<\alpha}d_{i}<\infty \quad\text{and}\quad D(\alpha)=\sum_{d_{i}\geq \alpha}(1-d_{i})<\infty.
\]
Let $0\le \eta <1$ be such that
\begin{equation}\label{La5}
C(1/2) - D(1/2) \equiv \eta \mod 1.
\end{equation}
By Lemma \ref{eqv}
\begin{equation}\label{La6}
\Lambda_N(\{d_i\}) \subset \bigcup_{k=0}^{N-1}
\bigg\{\boldsymbol \la \in (0,1)^N: \sum_{i=1}^N \la_i = k+\eta \bigg\}.
\end{equation}
By \eqref{La1} there exists $\ve>0$ such that for all $k=0,\ldots, N-1$ (exclude $k=0$ when $\eta=0$)
we have \eqref{fm2b} and \eqref{fm2c}.
Hence, by Lemma \ref{app} there exists a sequence $\{\tilde d_i\}_{i\in I}$ such that (i)--(v) hold. 

Let $\{\tilde d_i\}_{i=1}^M$ be a nonincreasing subsequence of $\{\tilde d_i\}$ consisting of terms in $(0,1)$. By Lemma \ref{decomp} we have $\Lambda_N(\{\tilde d_i\}_{i\in I})=\Lambda_N(\{\tilde d_i\}_{i=1}^M)$. By \eqref{La5} and Lemma \ref{app} there exists $K \in \N$ such that 
\[
\sum_{i=1}^M \tilde d_i = K +\eta.
\] 
Fix $k=0,\ldots,N-1$ and exclude $k=0$ when $\eta=0$. By Lemma \ref{fm} applied to $\{\tilde d_i\}_{i=1}^M$ there exists $\boldsymbol \mu \in [\ve,1-\ve]^N$ such that \eqref{fm2} holds.  Thus, by the Schur-Horn Theorem \ref{horn}
\begin{equation}\label{La8}
\Lambda_N(\{\tilde d_i\}) \cap \bigg\{\boldsymbol \la \in (0,1)^N: \sum_{i=1}^N \la_i = k+\eta \bigg\}
=\{ \boldsymbol \la \in (0,1)^N: \boldsymbol \mu \prec \boldsymbol \la \}.
\end{equation}

We claim that
\begin{equation}\label{La7}
\Lambda_N(\{d_i\}) \cap \bigg\{\boldsymbol \la \in (0,1)^N: \sum_{i=1}^N \la_i = k+\eta \bigg\} = 
\{ \boldsymbol \la \in (0,1)^N: \boldsymbol \mu \prec \boldsymbol \la \}.
\end{equation}
Indeed, take any $\boldsymbol \la\in (0,1)^N$ such that $\sum_{i=1}^N \la_i = k+\eta$. Let $0<\ve'<\ve$ be such that $\boldsymbol \la \in [\ve',1-\ve']^N$. By Lemma \ref{app} for $\ve'$ there exists a  sequence $\{\tilde d_i'\}_{i\in I}$ such that (i)--(v) hold. Moreover, by \eqref{app2}, $\boldsymbol \la \in \Lambda_N(\{d_i\}) \iff \boldsymbol \la \in \Lambda_N(\{\tilde d'_i\})$.  
Let $\{\tilde d_i'\}_{i=1}^{M'}$ be a nonincreasing subsequence of $\{\tilde d_i'\}$ consisting of terms in $(0,1)$. By properties (ii)--(v)  the assumptions \eqref{fm3b}--\eqref{fm3d} hold for sequences $\{\tilde d_i\}_{i=1}^{M}$ and $\{\tilde d_i'\}_{i=1}^{M'}$, resp. Therefore, by the second part of Lemma \ref{fm} we deduce that $\boldsymbol \la \in \Lambda_N(\{\tilde d'_i\}) \iff \boldsymbol \mu \prec \boldsymbol \la$. This shows \eqref{La7}. Combining \eqref{La6} with \eqref{La7} completes the proof of Theorem \ref{La}.
\end{proof}

\begin{remark}
The assumption \eqref{La1} is not a true limitation though the conclusion of Theorem \ref{La} needs to be modified accordingly. Indeed, suppose that $\{i\in I: d_i \in (0,1)\}$ is infinite. If $\{i\in I:1/2\le d_i <1\}$ is finite, then the operators $E$ in \eqref{LN} have finite rank since we can ignore the terms $d_i=0,1$ in light of Lemma \ref{decomp}. Let $k_0 \in \N_0$ and $0 \le \eta<1$ be such that
\[
\sum_{i\in I, \ d_i<1} d_i =k_0+\eta.
\]
Then, one can show that $\Lambda_N(\{d_i\})$ has exactly $\min(k_0+1,N)$ minimal elements when $\eta>0$ and $\min(k_0,N-1)$ minimal elements when $\eta=0$. This can be deduced in a similar way as Theorem \ref{La}. A similar result holds in the symmetric case when $\{i\in I: 0<d_i < 1/2\}$ is finite. Thus, Theorem \ref{La} can be extended to the case when we assume that $\{i\in I: d_i \in (0,1)\}$ is infinite. Finally, if the latter set is finite, then 
Lemma \ref{fm} and the Schur-Horn Theorem \ref{horn} alone give an extension of Theorem \ref{La}.
\end{remark}

We end the paper by illustrating Theorem \ref{La}.

\begin{ex} Let $\beta\in(0,1)$ and define the sequence $\{d_{i}\}_{i\in\Z\setminus\{0\}}$ by
\[
d_{i} = \begin{cases}1-\beta^{i} & i>0,
\\ 
\beta^{-i} & i<0.\end{cases}
\]
Let $\Lambda_N$ be the set defined by $\eqref{LN}$. Since $\{d_i\}$ is a diagonal of a projection we have $\Lambda_1=\varnothing$. Let $\boldsymbol \mu^k$, where $k=1,\ldots,N-1$, be the minimal elements of $\Lambda_N$ as in Theorem \ref{La}.  One can show that when $N=2$ we have 
\[
\boldsymbol \mu^1=
\begin{cases} (1-\frac{\beta}{1-\beta},\frac{\beta}{1-\beta}) & 0<\beta<1/3, 
\\
(1/2,1/2) &\text{otherwise}.
\end{cases}
\]
When $N=3$ we have $\boldsymbol \mu^2 =\mathbf 1 -\boldsymbol \mu^1$ and
\begin{align*}
\boldsymbol \mu^1 =
\begin{cases} (1-\frac{\beta}{1-\beta},\beta,\frac{\beta^2}{1-\beta}) 
& 0<\beta<\frac{3-\sqrt{5}}2 \approx 0.381966, \\
(\frac12 - \frac{\beta^2}{2(1-\beta)},\frac12 - \frac{\beta^2}{2(1-\beta)}, \frac{\beta^2}{1-\beta}) & \frac{3-\sqrt{5}}2 \le \beta < \frac{-1+\sqrt{13}}6 \approx 0.434259,
\\
(\frac13,\frac13,\frac13) & 
\text{otherwise}.
\end{cases}
\end{align*}
Likewise, when $N=5$ we have $\boldsymbol \mu^3 =\mathbf 1 -\boldsymbol \mu^2$, $\boldsymbol \mu^4 =\mathbf 1 -\boldsymbol \mu^1$, and
\begin{align*}
\boldsymbol \mu^1 &=
\begin{cases} (1-\frac{\beta}{1-\beta},\beta,\beta^2,\beta^3,\frac{\beta^4}{1-\beta}) 
& 0<\beta<\frac{3-\sqrt{5}}2, 
\\
(\frac12 - \frac{\beta^2}{2(1-\beta)},\frac12 - \frac{\beta^2}{2(1-\beta)}, \beta^2,\beta^3, \frac{\beta^4}{1-\beta}) & \frac{3-\sqrt{5}}2 \le \beta < \frac12,
\\
(\frac13 - \frac{\beta^3}{3(1-\beta)},\frac13 - \frac{\beta^3}{3(1-\beta)}, \frac13 - \frac{\beta^3}{3(1-\beta)}, \frac{\beta^3}{2(1-\beta)},\frac{\beta^3}{2(1-\beta)}) & \frac12 \le \beta < \frac{(45+\sqrt{2145})^{2/3}-2 \cdot 15^{1/3}}{15^{2/3} (45+\sqrt{2145})^{1/3}} \approx 0.560286,
\\
(\frac15,\frac15,\frac15,\frac15,\frac15) & \text{otherwise},
\end{cases}
\\
\boldsymbol \mu^2 &=
\begin{cases}
 (1-\frac{\beta^2}{1-\beta},1-\beta,\beta,\beta^2,\frac{\beta^3}{1-\beta}) 
& 0<\beta<\frac12,
\\
(\frac23-\frac{\beta^2}{3(1-\beta)},\frac23-\frac{\beta^2}{3(1-\beta)},\frac23-\frac{\beta^2}{3(1-\beta)},\frac{\beta^2}{2(1-\beta)},\frac{\beta^2}{2(1-\beta)})
&
\frac12 \le \beta< \frac25(-1+\sqrt 6) \approx 0.579796,
\\
(\frac25,\frac25,\frac25,\frac25,\frac25) & \text{otherwise}.
\end{cases}
\end{align*}

\end{ex}

%%%%% save for future
\begin{comment}
Suppose that $E$ is a self-adjoint operator on $\Hil$. The spectral theorem asserts that there exists a unique resolution of the identity $\pi$ on the Borel subsets of $\R$, called spectral measure, which satisfies
\[
E = \int_{\sigma(E)} \la d\pi(\la).
\]
 Define the positive  part of the operator $E$ by
\[
E_+:= \int_{[0,\infty)} \la d\pi(\la).
\]
The negative part is defined as $E_-=(-E)_+$.
\begin{lem}\label{two traces} Let $0<\alpha<B<\infty$ and $E$ be a self-adjoint operator with spectral measure $\pi$. Let $\{e_{i}\}_{i\in I}$ be an orthonormal basis for $\Hil$ and $d_{i}=\langle Ee_{i},e_{i}\rangle$.  Assume that $E_-$ and $(E-B)_+$ are trace class
and the series
\begin{equation*}C = \sum_{d_{i}<\alpha}d_{i}
\qquad\text{and}\qquad 
D = \sum_{d_{i}\ge \alpha}(B-d_{i})
\end{equation*}
are absolutely convergent. 
Then,
\begin{equation*}
K_1=\int_{(-\infty,\alpha)}  \la d\pi(\la), \qquad
K_2=\int_{[\alpha,\infty)} (\la-B) d\pi(\la).
\end{equation*}
are two trace class operators with $\sigma(K_1)\subset (-\infty,\alpha]$ and $\sigma(K_2)\subset (- \infty ,B-\alpha]$. Moreover, $E=BP+K_1+K_2$, where $P=\pi([\alpha,\infty))$ is an orthogonal projection onto a subspace $V \subset \Hil$ satisfying
\[
\overline{K_2(\Hil)}\subset V\subset K_1(\Hil)^{\bot}.
\]
\end{lem}
\end{comment}
%%%%%%%%%

\end{document}